\documentclass{amsart}

\usepackage[utf8]{inputenc}
\usepackage{amsfonts, amsthm, amssymb, mathtools,etoolbox,stmaryrd}
\usepackage{blindtext}
\usepackage[colorlinks=true, urlcolor=blue, linkcolor=blue, citecolor=blue]{hyperref}
\usepackage{tikz,tikz-cd}
\usepackage{array}
\usepackage{framed}
\usepackage{xcolor}
\usepackage{graphicx}
\usepackage{autobreak}
\usepackage[maxnames=10, sorting = nyt
]{biblatex}

\renewbibmacro{in:}{}
\tikzset{pullback/.style={minimum size=1.2ex,path picture={
\draw[opacity=1,black,-,#1] (-0.5ex,-0.5ex) -- (0.5ex,-0.5ex) -- (0.5ex,0.5ex);%
}}}

\theoremstyle{plain}
\newtheorem{theorem}{Theorem}[section]
\newtheorem{proposition}[theorem]{Proposition}
\newtheorem{lemma}[theorem]{Lemma}
\newtheorem{corollary}[theorem]{Corollary}
\newtheorem*{corollary*}{Corollary}

\theoremstyle{definition}
\newtheorem{example}[theorem]{Example}
\newtheorem{definition}[theorem]{Definition}
\newtheorem{remark}[theorem]{Remark}
\newtheorem{notation}[theorem]{Notation}

\title{Quotient toposes of discrete dynamical systems}
\author{Ryuya Hora, Yuhi Kamio}
\thanks{Graduate School of Mathematical Sciences, University of Tokyo. \url{hora@ms.u-tokyo}}
\thanks{College of Arts and Sciences, University of Tokyo. \url{emirp13@g.ecc.u-tokyo.ac.jp}}

\newcommand{\dq}[1]{``{#1}''}

\newcommand{\invmemo}[1]{}
\newcommand{\gen}[1]{\langle {#1}\rangle}
\newcommand{\N}{\mathbb{N}}
\newcommand{\Z}{\mathbb{Z}}
\newcommand{\Set}{\mathrm{Set}}
\newcommand{\Cat}{\mathrm{Cat}}
\newcommand{\Dw}[1]{\mathcal{D}({#1})}
\newcommand{\E}{\mathcal{E}}
\newcommand{\F}{\mathcal{F}}
\newcommand{\op}{\mathrm{op}}
\newcommand{\ob}[1]{\mathrm{ob}(#1)}
\newcommand{\ps}[1]{\Set^{{#1}^{\op}}}
\newcommand{\DD}{\ps{\N}}
\newcommand{\go}{\vartriangleright}
\newcommand{\goeq}{\trianglerighteq}
\newcommand{\og}{\vartriangleleft}
\newcommand{\ogeq}{\trianglelefteq}
\newcommand{\gog}{\simeq}
\newcommand{\G}{\mathbb{G}}
\newcommand{\Gc}{\mathbb{G}^{\mathrm{c}}}
\newcommand{\PQ}{\mathbf{PQ}}
\newcommand{\Q}{\mathbf{Q}}
\newcommand{\EQ}{\mathbf{EQ}}
\newcommand{\0}{\mathbf{0}}
\newcommand{\1}{\mathbf{1}}

\newcommand{\ph}{\varphi}
\newcommand{\C}{\mathcal{C}}
\newcommand{\D}{\mathcal{D}}
\newcommand{\Mr}[2]{\mathbb{\N}_{#1,#2}}
\newcommand{\T}[2]{\mathbb{T}_{#1,#2}}
\newcommand{\ml}{\ast}
\newcommand{\id}{\mathrm{id}}
\newcommand{\qi}{quasi-invertible}
\newcommand{\qa}{quasi-absorbing}
\newcommand{\Qi}{Quasi-invertible}
\newcommand{\Qa}{Quasi-absorbing}
\newcommand{\tn}{\otimes_{\N}}
\newcommand{\cN}{\overline{\N}}
\newcommand{\cpN}{\N^{\mathrm{div}}}
\newcommand{\NN}{\cN \times \cpN}
\newcommand{\rNN}{\N \times \cpN}
\newcommand{\Nsucc}{(\N, \mathrm{succ})}
\newcommand{\X}{\mathbb{X}}
\newcommand{\Y}{\mathbb{Y}}

\newcommand{\scvalue}{1.1}
\newcommand{\w}{w}

\newcommand{\cP}{\overline{\mathbb{P}}}
\DeclareMathOperator*{\colim}{colim}
\DeclareMathOperator{\Lcm}{lcm}
\DeclareMathOperator{\Gcd}{gcd}

\newcommand{\Image}[1]{\mathrm{Im}(#1)}
\newcommand{\Ideal}{\mathrm{Ideal}}

\newcommand{\ADJ}[4]
    {
    \begin{tikzcd}[ampersand replacement = \&, column sep = small]
        {#1}
        \ar[rr, shift right=1.3ex, "{#2}"']
        \&\perp\&
        {#3}
        \ar[ll, shift right=1.3ex,"{#4}"']
    \end{tikzcd}
    }
\newcommand{\arst}{-{Stealth[length=2mm]}}
\DeclarePairedDelimiter{\abs}{\lvert}{\rvert}

\begin{document}
\definecolor{shadecolor}{gray}{0.9}
\maketitle

\begin{abstract}
Lawvere's open problem on quotient toposes has been solved for boolean Grothendieck toposes but not for non-boolean toposes. As a simple and non-trivial example of a non-boolean topos, this paper provides a complete classification of the quotient toposes of the topos of discrete dynamical systems, which, in this context, are sets equipped with an endofunction. This paper also offers an order-theoretic framework to address the open problem, particularly useful for locally connected toposes.

Our result is deeply related to monoid epimorphisms. At the end of this paper, utilizing the theory of lax epimorphisms in the $2$-category $\Cat$, we explain how (non-surjective) monoid epimorphisms from $\N$ correspond to 
(non-periodic) behaviors in discrete dynamical systems.
\end{abstract}
\tableofcontents

\section{Introduction}\label{SectionIntroduction}
\invmemo{Something for those who don't know topos theory}
\subsection{Main problem and its answer}
The main theorem of this paper provides a complete classification of classes of \emph{discrete dynamical systems} (a pair of a set $X$ and an endofunction $f\colon X \to X$) that are closed under finite limits and small colimits. There are numerous such classes, including those for which:
\begin{itemize}
\item Every state is in a loop. 
\[\forall x \in X,\  \exists n>0,\ f^{n}(x)=x.\]
\item Every state is eventually fixed. 
\[\forall x \in X,\  \exists n>0,\ f^{n+1}(x)=f^{n}(x).\]
\item $f$ is bijective.
\item $f$ is a bijection when restricted to $\Image{f}$.
\item Every state enters a loop within two steps, where the period of the loop has no square factors.
\[\forall x \in X,\ \exists n>0,\ (f^{n+2}(x) = f^{2} (x)) \land (\forall p:\text{prime}\ p^2 \nmid n)\]
\end{itemize}
The goal of this paper is to describe these classes uniformly and clarify the background mathematical structures.

Our main theorem (Theorem \ref{TheoremMainTheorem}) states that these classes are in one-to-one correspondence with ideals of the product poset $\rNN$, where $\N$ denotes the usual poset of natural numbers and $\cpN$ denotes the poset of all natural numbers (including $0$) with the divisibility order. Moreover, we present two specific construction methods (Proposition \ref{PropositionConstructionPrimeQuotients} and Proposition \ref{PropositionEventualBijConstruction}) and prove that all classes are produced by these constructions (Corollary \ref{CorollaryElementaryInAdvance}).

The difficulty of this problem lies in the analysis of non-periodic behaviors of states. In the study of discrete dynamical systems, examining the behavior of states, especially loops, is crucial (see, for example, \cite[][Theorem 5]{scandolo2023covariant}).
To prove our main theorem, we need to generalize the classically studied quantity, the \dq{time to enter a loop,} even for states that do not enter loops (Definition \ref{DefinitiongeneralizedHeight}).

\subsection{Relationship with monoid theory}

This problem on discrete dynamical systems involves monoid theory. Discrete dynamical systems can be thought of as actions on sets of the monoid $\N$.
Therefore, the properties of the monoid $\N$ are deeply intertwined. 

The most crucial relationship lies in the classification of monoid epimorphisms from $\N$, which are not necessarily surjections! (The simplest counterexample might be $\N \to \Z$.) 
Roughly speaking, the non-trivial \dq{non-periodic behaviors}, which we mentioned in the last subsection, correspond to non-surjective epimorphisms, while the \dq{usual behaviors} of a state correspond to surjective homomorphisms from $\N$.
Although this structure underlies the entire paper, we will postpone the discussion on this point until section \ref{SectionWhere} and Appendix \ref{AppendixMonoidEpimorphismsFromN}, after the proof of the main theorem. 

As we will mention in section \ref{SectionConclusion}, some topological monoids are also involved, including $\Z_p, \hat{\Z}$, and the one-point compactification of $\N$. This relationship with topological monoids is due to \cite{rogers2023toposes}.


\subsection{Lawvere’s open problem on quotient toposes}
\subsubsection{The open problem and current situations}
This fun puzzle is actually motivated by the first problem of Lawvere’s open problems in topos theory \cite{OpenLawvere}\footnote{Presently (February 2024), Lawvere's homepage is unavailable. However, the document can be viewed through a link provided on nLab \url{https://ncatlab.org/nlab/show/William+Lawvere}.}, \dq{Quotient toposes.} (The latter part of) the problem seeks a complete description of quotients of Grothendieck toposes (i.e., a co-reflective subcategory of a Grothendieck topos closed under finite limits). The main theorem of this paper solves this problem for the topos of discrete dynamical systems.

Research on this open problem remains underdeveloped, with much still unknown. What is known includes the description of atomic quotients \cite{henry2018localic}, the description of essential quotients of presheaf toposes \cite{el2002simultaneously}, the (external) description of hyperconnected quotients \cite{rosenthal1982quotient}, and the (internal) description of hyperconnected quotients by the first author \cite{hora2023internal}.

However, even using any of these, it is non-trivial to describe the quotients of even a very simple topos, the topos of discrete dynamical systems!

Approaching this underdeveloped open problem, obtaining a simple yet non-trivial example is crucial for gaining theoretical insights. In this sense, classifying the quotients of the very simple topos and revealing their underlying mathematical structure could be an important step towards solving the open problem.

\subsubsection{Why discrete dynamical systems?}
Among many Grothendieck toposes, why do we choose the topos of discrete dynamical systems? 
One reason is 
the fact that the topos of discrete dynamical systems is non-boolean (see Example \ref{ExampleSubobjectClassifier}).
This is because the problem for boolean Grothendieck toposes has already been solved by the first author in \cite{hora2023internal}.

Another answer is: it's simple and deep!
Despite its innocent and straightforward definition of discrete dynamical systems, it remains a mysterious object with many puzzles, like the Collatz conjecture. The two-sidedness of its simplicity and richness might be the reason why Lawvere and Schanuel chose it as one of the main topics in their introductory book on category theory \cite{lawvere2009conceptual}.
In addition to its relationship with monoid theory and dynamical systems, recent research \cite{tomasic2020topos} explores the application of the topos on difference algebra and algebraic geometry.

\subsubsection{Future works that our results suggest}
Although our results are about the specific topos, one of our aims is to gain insights for solving Lawvere's open problem. Here, we'll touch upon some of these insights.

First, we clarified the relationship with monoid epimorphisms. Actually, this is the concept of lax epimorphisms in $2$-category theory (see \cite{adamek2001functors, el2002simultaneously, nunes2022lax}) rephrased in monoid-theoretic terminology (see Proposition \ref{PropositionEquivalenceOfEpiLaxEpi}). 
Our results, especially Corollary \ref{CorollaryEssentialsAreDense}, suggest that the study of lax epimorphisms, which correspond to essential quotients, could play a crucial role in solving the open problem. To pursue this point of view, the topos of directed graphs might be the next non-trivial example (see Remark \ref{RemarkDirectedGraph}).

Second, 
our main theorem might suggest the possibility of an extended theory of \cite{hora2023internal}, which originally gives a way to classify all \emph{hyperconnected quotients}, but not all quotients. 
This is because our conclusion, the one-to-one correspondence with ideals of semilattices $\rNN$, is very similar to the classification of the hyperconnected quotients given by \cite{hora2023internal}, which states the correspondence with ideals of $(\N \times \cpN _{>} )\cup \{\infty\}$ (Corollary \ref{CorollaryHyperconnectedQuotientsOfDD}).

\subsection{Technical Novelty, the generative order}
The method we will introduce is applicable to quotients of general toposes, not only the topos of discrete dynamical systems. Our idea is simple: reduce the computation of quotients (classes of objects) to the computation of objects. We define a preorder among objects of the topos, which we call the \emph{generative order}, to be equivalent to the poset of quotients and inclusions (in the sense of Remark \ref{RemarkQuotientToposesAsClassesOfObjects}). Although its definition is, in some sense, tautological, it provides a highly convenient framework for concrete computations.

In particular, for locally connected toposes (including the topos of discrete dynamical systems), the discussion can be reduced to the generative order on the subcategory of connected objects. These technical preparations will enable us to reduce our problem to the computation of an order structure among connected discrete dynamical systems.

\subsection{Preliminaries and notations}
This paper requires knowledge of category theory, including:
\begin{enumerate}
    \item Concepts of categories, functors, and natural transformations.\label{itemConceptsCFN}
    \item Concepts of limits and colimits.\label{itemConceptsLim}
    \item Concept of adjunction.\label{itemConceptsAdj}
    \item Basic concepts of topos theory.\label{itemConceptsTopos}
\end{enumerate}

For (\ref{itemConceptsCFN}), (\ref{itemConceptsLim}), and (\ref{itemConceptsAdj}), the reader can refer to basic textbooks on category theory \cite{awodey2010category, mac2013categories, riehl2017category}. For (\ref{itemConceptsTopos}), see subsection \ref{subsectionDefQuotients} and Appendix \ref{AppendixGeometricMorphisms}, or the references cited therein.

\begin{notation}\label{NotationNNN}
    We adopt the following notations:
    \begin{itemize}
        \item $\N$ denotes the set of all non-negative integers.
        \[\N = \{0,1,2,\dots\}\]
        We also call $\N$ the set of natural numbers.
        When we regard it as a category, $\N$ denotes the one-object category. 
        As a poset, the same notation denotes the usual totally ordered set $0<1<2<3<\dots$.
        \item $\cN$ denotes the set of all natural numbers with a formal symbol $\infty$.
        \[\cN = \N \cup \{\infty\}\] This set is equipped with the intended total order $0<1<2<3<\dots<\infty$.
        \item $\cpN$ denotes the poset of all natural numbers $\{0,1,2,3, \dots\}$ equipped with the divisibility partial order.
        For example, in this poset, we have $2<4<12<0$ and $2\not < 3$.
    \end{itemize}
\end{notation}
\section{Discrete dynamical systems}
\subsection{Definition and Examples}
First, we define the notion of discrete dynamical systems and see some theoretically important examples.
\begin{definition}
    A discrete dynamical system $\X=(X,f)$ is a set $X$ equipped with an endofunction $f\colon X\to X$. A morphism $h\colon (X,f)\to (Y,g)$ is a function $h\colon X \to Y$ such that 
    \[
    \begin{tikzcd}
        X\ar[r,"f"]\ar[d,"h"]&X\ar[d,"h"]\\
        Y\ar[r,"g"]&Y
    \end{tikzcd}
    \]
    commutes.
\end{definition}

Equivalently, a discrete dynamical system is a presheaf over the additive monoid of natural numbers $\N$ regarded as a category with one object. A morphism of discrete dynamical systems is a natural transformation. For those reasons, the category of discrete dynamical systems is denoted by $\DD (\cong \Set^{\N})$.

Notice that other articles call them by other names. For example, they are called \emph{sets-with-an-endomap}, \emph{dynamical systems} or \emph{automata} in \cite{lawvere2009conceptual}, and \emph{difference sets} in \cite{tomasic2020topos}. From the viewpoint of monoid actions, they are also called \emph{$\N$-sets}.

\begin{example}[Canonical endomorphism]\label{ExampleCanonicalEndo}
    For a discrete dynamical system $\X=(X,f)$, the endofunction $f$ on  $X$ is an endomorphism on 
    $\X$
    \[
    \begin{tikzcd}
        X\ar[r,"f"]\ar[d,"f"]&X\ar[d,"f"]\\
        X\ar[r,"f"]&X.
    \end{tikzcd}
    \]
    We will utilize this endomorphism repeatedly for various constructions, especially in subsection \ref{SubsectionConnectedIsom}.
\end{example}

\begin{example}[Free dynamical system]\label{ExampleNsucc}
    A prototypical example of a discrete dynamical system is $\Nsucc$, which is the free object $F{1}$ generated by a singleton $1$ via the free-forgetful adjunction.
    \[\ADJ{\DD}{U}{\Set}{F}\]
    By the density theorem, every discrete dynamical system is canonically a colimit of a diagram in which every object is $(\N, \mathrm{succ})$. This fact is used later (Lemma \ref{LemmaNsuccisMaximum}).
\end{example}

\begin{example}[Monoid]\label{ExampleMonoidDD}
    For a monoid $M$ and an element $a\in M$, we can construct a discrete dynamical system $(M, -\ml a)$. This paper will study a certain class of monoids, and this construction is the essential bridge between monoids and discrete dynamical systems.

    This discrete dynamical system captures some properties of the element $a$. For example, $-\ml a$ is injective if and only if $a$ is right cancellative. $-\ml a$ is surjective if and only if $a$ has a left inverse. In Appendix \ref{AppendixMonoidEpimorphismsFromN}, we will utilize the properties of an element described in terms of the induced discrete dynamical system.
    \end{example}


\begin{example}[As a directed graph]\label{ExampleVisualization}
    A discrete dynamical system $\X=(X,f)$ may be visualized as a directed graph whose vertices have out-degree $1$. The vertices are the elements of $X$ and an edge $x\to f(x)$ is drawn for each $x\in X$. For example, $\X=(\Z/4\Z, -\times 2)$ (see \ref{ExampleMonoidDD}) is visualized as Figure \ref{PictureOfSample}.
    \begin{figure}[ht]
        \begin{shaded}
        \centering
        \begin{tikzpicture} [scale = \scvalue]
            \fill[black] (-1,2) circle (0.06) node[left]{$1$};
            \fill[black] (1,2) circle (0.06) node[right]{$3$};
            \fill[black] (0,1) circle (0.06) node[left=5pt]{$2$};
            \fill[black] (0,0) circle (0.06) node[left = 5pt]{$0$};
            \draw[black, thick] (0,-0.30) circle (0.30) node{};
            \draw[black, thick,\arst](-1,2)--(0,1);
            \draw[black, thick,\arst](0,1)--(0,0);
            \draw[black, thick,\arst] (1,2)--(0,1);
            \draw[black, thick,\arst] (-0.01,-0.002)--(0,0);
        \end{tikzpicture}
        
        \caption{Associated graph of $(\Z/4\Z, - \times 2)$}
        \label{PictureOfSample}
        \end{shaded}
    \end{figure}
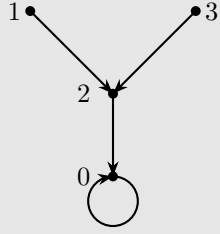
    \end{example}

    \begin{remark}\label{RemarkDirectedGraph}
    This visualization also arises from a kind of nerve and realization adjunction. Furthermore, it proves that $\ps{\N}$ is a quotient topos of the topos of directed graphs. Let $\mathrm{Par}$ denote the parallel morphism category, which looks like
    \[
    \begin{tikzcd}
        \bullet\ar[r,shift right,"t"']\ar[r,shift left,"s"]& \bullet.
    \end{tikzcd}
    \]
    The category of directed graphs is equivalent to the presheaf category over $\mathrm{Par}$. The localization of the category $\mathrm{Par}$ by the set of morphisms $\{t\}$ is equivalent to the monoid $\N$ regarded as a one-object category. The associated functor $p\colon \mathrm{Par} \to \N$ induces a fully faithful functor $-\circ p \colon \DD\to \ps{\mathrm{Par}}$, because $p$ is a localization (and hence a lax epimorphism \cite{adamek2001functors}). 
    Drawing the image of a discrete dynamical system under the functor in the usual way produces the visualization adopted above.
    
    Furthermore, this induced functor is a part of connected geometric morphism
    \[\ADJ{\ps{\mathrm{Par}}}{\mathrm{Ran}_{p}}{\DD.}{-\circ p}\]
    Consequently, we can regard $\DD$ as a quotient topos of $\ps{\mathrm{Par}}$ consisting of a directed graph whose vertices have out-degree $1$. 
    
    In this paper, we will classify the quotients of $\DD$. However, it is merely a subproblem of a much more difficult question, quotients of the directed graph topos!
    \end{remark}


\begin{example}[Subobject classifier]\label{ExampleSubobjectClassifier}
Since the category of discrete dynamical systems is a presheaf topos, it has a subobject classifier $\Omega$, which looks like Figure \ref{PictureOfSubobjectclassifier}. Intuitively, this subobject classifier describes the \dq{time until a subobject is reached}. We will utilize this in Section \ref{SectionMainTheorem}.
    \begin{figure}[ht]
    \begin{shaded}
    \centering
        \begin{tikzpicture} [scale = \scvalue]
            \fill[black] (0,0) circle (0.06) node[above]{$0$};
            \fill[black] (1,0) circle (0.06) node[above]{$1$};
            \fill[black] (2,0) circle (0.06) node[above]{$2$};
            \fill[black] (3,0) circle (0.06) node[above]{$3$};
            \fill[black] (4,0) circle (0.06) node[above]{$4$};
            \fill[black] (5,0) circle (0.00) node[]{$\cdots$};
            \fill[black] (6,0) circle (0.06) node[above]{$\infty$};
            \draw[black, thick,\arst](1,0)--(0,0);
            \draw[black, thick,\arst](2,0)--(1,0);
            \draw[black, thick,\arst](3,0)--(2,0);
            \draw[black, thick,\arst](4,0)--(3,0);
            \draw[black, thick,\arst](4.5,0)--(4,0);
            \draw[black, thick] (0,-0.3) circle (0.3) node[]{};
            \draw[black, thick,\arst](0+0.1,0-0.03)--(0,0);
            \draw[black, thick] (6,-0.3) circle (0.3) node[]{};
            \draw[black, thick,\arst](6+0.1,0-0.03)--(6,0);
        \end{tikzpicture} 
    \caption{Associated graph of the subobject classifier $\Omega$}
    \label{PictureOfSubobjectclassifier}
    \end{shaded}
\end{figure}
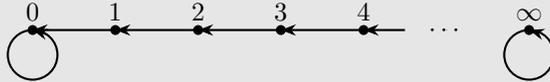
\end{example}

\subsection{Connectedness}\label{SubsectionConnectednessDD}
Our classification of quotient toposes is based on the notion of connectedness (see subsection \ref{subsectionLocallyConnectedTopos}). 
\begin{definition}[Connectedness]\label{DefinitionConnectedDD}
    A discrete dynamical system $\X=(X,f)$ is connected if $\X$ satisfies the following equivalent conditions:
    \begin{enumerate}
        \item $X\neq \emptyset$, and if $X$ is divided into a binary coproduct $X = S \coprod T$ of two subobjects, then either $S$ or $T$ is empty and the other is equal to $X$.
        \item $X\neq \emptyset$, and for any $x,y\in X$, there exist natural numbers $n,m$ such that \[f^n (x)=f^{m}(y).\] \label{ConditionConcreteConnectedness}
        \item The represented functor $\DD (\X,-)\colon \DD \to \Set$ preserves small coproducts.
    \end{enumerate}
\end{definition}
\invmemo{Proof of the equivalence}
The equivalence of the above three conditions is not hard to prove.
More intuitively, $\X$ is connected if and only if it is \dq{visually connected}, i.e., its associated directed graph Example \ref{ExampleVisualization} is connected. 

We will later utilize the fact that every discrete dynamical system is a coproduct of connected ones. This is equivalent to saying that $\DD$ is a \emph{locally connected topos}. 
This fact constitutes the core of our approach to the classification of quotients of $\DD$ (see subsection \ref{subsectionLocallyConnectedTopos}).


\begin{example}[Collatz map]\label{ExampleCollatz}
The Collatz map, which appears in the Collatz conjecture, is one of the most fascinating and enigmatic examples of discrete dynamical systems. This map is defined on the set of positive integers $\N_{>}$ as an endofunction, where even numbers are divided by $2$ and odd numbers are multiplied by $3$ and $1$ is added to the result (Figure \ref{PictureOfCollatz2}).
\[
n\mapsto
\begin{cases}
    n/2 & (n\text{: even})\\
    3n+1 & (n\text{: odd})
\end{cases}
\]
The question of whether this discrete dynamical system is connected or not remains a famous unsolved problem.
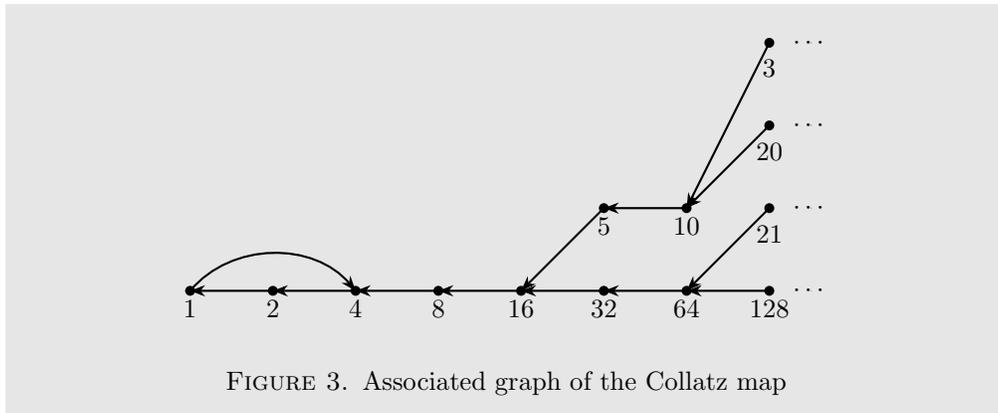
\begin{figure}[ht]
    \begin{shaded}
    \centering
        \begin{tikzpicture} [scale = \scvalue]
            \fill[black] (0,0) circle (0.06) node[below]{$1$};
            \draw[black, thick,\arst, bend right] (0,0)to [out=50,in=-230](2,0);
            \fill[black] (1,0) circle (0.06) node[below]{$2$};
            \draw[black, thick,\arst](1,0)--(0,0);
            \fill[black] (2,0) circle (0.06) node[below]{$4$};
            \draw[black, thick,\arst](2,0)--(1,0);
            \fill[black] (3,0) circle (0.06) node[below]{$8$};
            \draw[black, thick,\arst](3,0)--(2,0);
            \fill[black] (4,0) circle (0.06) node[below]{$16$};
            \draw[black, thick,\arst](4,0)--(3,0);
            \fill[black] (5,0) circle (0.06) node[below]{$32$};
            \draw[black, thick,\arst](5,0)--(4,0);
            \fill[black] (5,1) circle (0.06) node[below]{$5$};
            \draw[black, thick,\arst](5,1)--(4,0);
            \fill[black] (6,0) circle (0.06) node[below]{$64$};
            \draw[black, thick,\arst](6,0)--(5,0);
            \fill[black] (6,1) circle (0.06) node[below]{$10$};
            \draw[black, thick,\arst](6,1)--(5,1);
            \fill[black] (7,0) circle (0.06) node[below]{$128$};
            \draw[black, thick,\arst](7,0)--(6,0);
            \fill[black] (7,1) circle (0.06) node[below=3pt]{$21$};
            \draw[black, thick,\arst](7,1)--(6,0);
            \fill[black] (7,2) circle (0.06) node[below=3pt]{$20$};
            \draw[black, thick,\arst](7,2)--(6,1);
            \fill[black] (7,3) circle (0.06) node[below=3pt]{$3$};
            \draw[black, thick,\arst](7,3)--(6,1);
            
            \fill[black] (7.5,0) circle (0) node{$\cdots$};
            \fill[black] (7.5,1) circle (0) node{$\cdots$};
            \fill[black] (7.5,2) circle (0) node{$\cdots$};
            \fill[black] (7.5,3) circle (0) node{$\cdots$};
        \end{tikzpicture} 
    \caption{Associated graph of the Collatz map}
    \label{PictureOfCollatz2}
    \end{shaded}
\end{figure}
\end{example}

    

    

    


\subsection{Period and Height}
In order to analyze a discrete dynamical system in detail, it is necessary to classify the behavior of each element. In particular, focusing on periodic behaviors, we introduce the following two quantities.

\begin{definition}[Period and Height of an element]\label{DefinitionPeriodHeightElement}
    For a discrete dynamical system $\X = (X,f)$ and an element $x\in X$, if the sequence
    \[x, f(x), f^2(x), f^3 (x) ,\cdots\]
    eventually becomes periodic, call its minimum period the \textit{period of $x$}, and the number of first exceptions before it becomes periodic, the \textit{height of $x$}. If the sequence is never periodic, the period of $x$ is defined to be $0$, and the height is not defined (but will be defined in Definition \ref{DefinitiongeneralizedHeight}).
\end{definition}
In other words, the pair of the height $a$ and the period $b$ satisfies $f^{a+b}(x)=f^{a}(x)$, and it is minimum among such pairs. For example, for the Collatz map, 
\[\underbrace{10 \mapsto 5 \mapsto 16 \mapsto 8}_{\text{first }4 \text{ exceptions}}\mapsto \underbrace{ 4 \mapsto 2 \mapsto 1}_{\text{period }3} \mapsto \underbrace{4 \mapsto 2 \mapsto 1}_{\text{period }3} \mapsto \dots\]
the element $10\in \N_{>}$ has the height $4$ and the period $3$.
For the discrete dynamical system $\Nsucc$, every element has the period $0$, and its height is not (yet) defined.



\section{List of all quotients}\label{SectionConclusion}
In this section, we will state our conclusion.
We will define the notion of quotients of a topos in subsection \ref{subsectionDefQuotients} and list all quotients of the topos $\DD$ without proofs in the remaining subsections. The proofs will be given in Section \ref{SectionMainTheorem}.
\invmemo{Write it later}

\subsection{Definition of quotients}\label{subsectionDefQuotients}
In this subsection, we briefly recall some notions of toposes. For those who are not familiar with topos theory, it will not cause any serious problem if one just skips this subsection and regards a quotient of $\DD$ as a class of discrete dynamical systems that is closed under finite limits and small colimits (see Corollary \ref{CorollaryDDprequotientIsQuotient}). 

\begin{definition}[topos]\label{DefinitionTopos}
    A topos is a cartesian closed and finitely complete category with a subobject classifier.
\end{definition}
For details of the definition of topos, one can refer to basic textbooks on topos theory \cite{borceux1994handbookv3, goldblatt1984topoi, johnstone2002sketchesv1, johnstone2014topos, maclane1994sheaves}. However, for the purpose of this paper, only a little knowledge of toposes is needed. The most important one is the fact that every presheaf category is a topos, and hence $\DD$ is a topos.

A \emph{Grothendieck topos} is defined as a category that is equivalent to the sheaf category over a small site. Details of this notion are also found in the above textbooks. A Grothendieck topos is a topos in the sense of Definition \ref{DefinitionTopos}. Furthermore, it is locally small and complete and cocomplete. Every presheaf topos is a Grothendieck topos, and hence so is $\DD$.


\begin{definition}[Quotients]
    A quotient of a topos $\E$ is a full subcategory $Q$, whose inclusion functor preserves finite limits (i.e., is left exact) and has a right adjoint.
\[
\begin{tikzcd}[column sep = 10pt]
    \E \ar[rr, shift right=6pt]&\rotatebox{90}{$\vdash$}&Q.\ar[ll, shift right =6pt, hookrightarrow]
\end{tikzcd}
\]
\end{definition}
In topos theoretic terminology, a quotient is usually defined as a \emph{connected geometric morphism} from $\E$ (Definition \ref{DefinitionConnectedGM}). (For those who know topos theory, notice that $Q$ is also a topos, since $Q$ is the category of coalgebras of a left exact comonad.) We adopt the above definition just to make it available to readers who are not familiar with topos theory. For quotients and connected geometric morphisms, see \cite[C1.5 of][]{johnstone2002sketchesv2}.

\begin{remark}[Identification of quotient toposes]\label{RemarkIdentificationOfQuotiens}
By definition, a quotient topos is a full subcategory, but what we are interested in (in our context) is the essential image of its embedding. In other words, two quotient toposes whose embeddings have the same essential image are not distinguished in this paper. (This condition is equivalent, in terms of geometric morphisms, to the existence of an equivalence between the codomains of the corresponding two connected geometric morphisms, and further, that this equivalence commutes up to natural isomorphism.)
\end{remark}

\begin{remark}[Quotient toposes and classes of objects]\label{RemarkQuotientToposesAsClassesOfObjects}
    Based on the above remark, we sometimes treat a quotient topos (or, more generally, a full subcategory) as the class of objects that belong to the essential image of the embedding functor. For example, when we speak of the inclusion relations of quotient toposes, we refer to the inclusion relations of the corresponding classes of objects.
\end{remark}



To reinforce the understanding of the definition, we give an example of a topos where the classification of quotients is very simple. For other examples and more detailed discussions, see section \ref{sectionGenerativeOrder}.

\begin{example}[Quotients of the topos of sets]\label{ExampleQuotientSets}
    There is only one quotient for the topos of sets $\Set$. 
    Let $Q$ be an arbitrary quotient of $\Set$. As easily shown, $Q$ is closed under finite limits and small colimits (Lemma \ref{LemmaQuotientIsPrequotient}). 
    Therefore, the quotient $Q$ contains a singleton, which is the terminal object of $\Set$. Furthermore, since $Q$ is closed under taking small coproducts, 
    $Q$ contains every set (up to bijection). Consequently, the quotient $Q$ is $\Set$ itself.
\end{example}
\invmemo{We can write more.}

All of the following three examples of quotients are explained in \cite{johnstone2002sketchesv1} and \cite{johnstone2002sketchesv2}.
\begin{example}[Connectedness]\label{ExampleQuotientAsConnectedness}
    Part of the importance of the concept of a quotient lies in its close relationship with connectedness. A Grothendieck topos is said to be \emph{connected}, if it contains the category of sets $\Set$ as a quotient. (In topos theoretic terminology,  it is usually claimed that \dq{the unique geometric morphism to $\Set$ is connected.})

    This condition captures the usual connectedness of topological spaces and categories. The sheaf topos over a topological space $X$ is connected as a topos, if and only if $X$ is connected as a topological space. The presheaf topos over a small category $\C$ is connected as a topos, if and only if $\C$ is connected as a category.
\end{example}

\begin{example}[Bijective on objects and full functor]\label{ExampleBOFquotient}
 If a functor $F\colon \C \to \D$ between two small categories $\C, \D$ is bijective on objects and full, then the induced precomposition functor
 \[\ps{\C}\hookleftarrow \ps{\D}\]
 is fully faithful and defines a quotient. For example, a surjective monoid homomorphism $p\colon M \to N$ induces a quotient 
  \[\ps{M}\hookleftarrow \ps{N}.\]

  However, the condition \dq{bijective on objects and full} is too strong just for inducing a quotient. The sufficient and necessary condition is studied in \cite{adamek2001functors}. We will use this result in Section \ref{SectionWhere}.
\end{example}

\begin{example}[Topological Groups]\label{ExampleTopologicalGroupQuotient}
    For a topological group $G$, its continuous action topos $G$-$\Set$ is a quotient of the discrete actions topos of $G$, via the canonical embedding functor. It is known that every quotient of the group action topos is induced by a topology on the group (for example, see \cite{hora2023internal}).

    This topological construction can be extended to topological monoids (\cite{rogers2023toposes}). However, for monoid action toposes including $\DD$, not all quotients are necessarily induced from topology. 
    In fact, in \cite{rogers2023toposes}, the topos $\DD$ is used to exemplify that even a \emph{hyperconnected} quotient may not be induced by a topology.
\end{example}


\subsection{Quotients via Prime numbers}\label{subsectionPrime}

In this subsection, we introduce a construction of quotients of $\DD$, using prime numbers, and give a few examples of quotients.
\begin{notation}
    Let 
     $\cP$ denote the set of all prime numbers with a formal symbol $\infty$
    \[\cP=\{2,3,5,\dots\}\cup \{\infty\}.\]
\end{notation}

As explained in Notation \ref{NotationNNN}, $\cN$ denotes the set of all non-negative integers with the formal symbol $\infty$
 \[\cN = \{0,1,2,\dots\} \cup \{\infty\}.\]

\begin{proposition}[Quotients via Prime numbers]
\label{PropositionConstructionPrimeQuotients}For a function $\alpha\colon \cP \to \cN$, the full subcategory of $\DD$ that consists of a discrete dynamical system $\X=(X,f)$ that satisfies the following conditions is a quotient.

\textbf{Condition}: For every $x\in X$,
        \begin{enumerate}
            \item for any prime number $p$, the period of $x$ is non-zero and divided by $p$ at most $\alpha(p)$ times.
            \item the height of $x$ is at most $\alpha(\infty)$.
        \end{enumerate}
\end{proposition}
\begin{proof}
    As we will see in Section \ref{SectionWhere}, this follows from a general theorem in the author's paper \cite{hora2023internal}. However, a direct proof is not hard. The right adjoint $R$ of the inclusion functor is given by defining $RX$ to be the subset of elements that satisfies the above two conditions.
    
    For the left exactness of the inclusion, it is enough to see that the above condition is inherited by finite limits. The condition is inherited by subobjects, in particular to equalizers. For a finite product, the period of $(x_1, \dots x_n)\in X_1 \times \dots \times X_n$ is given by the least common multiple of the periods of each component, and the height is given by the maximum height. 
\end{proof}

\begin{remark}[Hyperconnected quotients]
    In the above proof, we have proven that the constructed quotient is closed under taking subobjects. A quotient with this subobject-closed property is called \emph{hyperconnected quotient} \cite{johnstone1981factorization} and is particularly important in topos theory (see \cite[A4.6 of ][]{johnstone2002sketchesv1}). 
    
    The classification method for hyperconnected quotients is known \cite{hora2023internal}. The above proposition is due to calculations using a known theorem rather than an insight. See Section \ref{SectionWhere} for more details.
\end{remark}

In the remaining part of this subsection, we will see some examples of quotients, constructed by Proposition \ref{PropositionConstructionPrimeQuotients}.

\begin{example}[Quotient of trivial systems]\label{ExampleTrivialSystems}
First we give a trivial example of a quotient.
    Defining $\alpha\colon \cP \to \cN$ by
    \[
    \alpha (q) = 0
    \]
    the associated quotient consists of the identity functions $\id_{X}\colon X\to X$. This quotient is equivalent to the topos of sets $\Set$. This fact correspond to the \emph{connectedness} of the topos $\DD$ and the category $\N$ (see \cite[C1.5 of ][]{johnstone2002sketchesv2}).
\end{example}

\begin{example}[Quotient of loops]\label{ExampleLoopsProfiniteCompletion}
    Defining $\alpha\colon \cP \to \cN$ by
    \[
    \alpha (q) = 
    \begin{cases}
        0 & (q=\infty)\\
        \infty & (q< \infty),
    \end{cases}
    \]
    the associated quotient consists of coproducts of loops.
    In other words, for every discrete dynamical system $\X=(X,f)$ in the quotient, for any $x\in X$ there exists $m>0$ such that $f^{m}(x)=x$. This quotient is equivalent to the topos of continuous actions of a topological group $\hat{\Z}$ of profinite integers.
\end{example}

\begin{example}[Eventually fixed systems]\label{ExampleEventualyFixed}
    A discrete dynamical system $\X = (X,f)$ is \emph{eventually fixed}, if for every $x\in X$, there exists $n\in \N$ such that $f^{n} (x)$ is a fixed point of $f$. The full subcategory of $\DD$ that consists of all eventually fixed discrete dynamical systems is a quotient, constructed by 
    \[
    \alpha (q) = 
    \begin{cases}
        \infty & (q=\infty)\\
        0 & (q<\infty).
    \end{cases}
    \]
    This quotient is equivalent to the topos of topological monoid actions of the one-point compactification of the discrete monoid $\N$.
\end{example}

\begin{example}[Eventually periodic systems]\label{ExampleEventualyPeriodic}
    A discrete dynamical system $\X = (X,f)$ is \emph{eventually periodic}, if for every $x\in X$, there exists $n\geq 0$ and $m>0$ such that $f^{n} (x)= f^{n+m}(x)$. The full subcategory of $\DD$ that consists of all eventually periodic discrete dynamical systems is a quotient, constructed by 
    \[
    \alpha (q) = \infty.
    \]
\end{example}


\begin{example}[$\Z/12 \Z$-actions]\label{ExampleTwelveLoop}
    Defining $\alpha\colon \cP \to \cN$ by
    \[
    \alpha (q) = 
    \begin{cases}
        2 & (q=2)\\
        1 & (q=3)\\
        0 & (\text{others}),
    \end{cases}
    \]
    the associated quotient consists of coproducts of loops whose periods are divisors of $2^2\cdot 3 = 12$. This quotient is equivalent to the topos of $\Z/12\Z$-actions. This construction also works for any positive integers other than $12$.
\end{example}

\begin{example}[$7$-adic number actions]\label{Examplepadic}
    Defining $\alpha\colon \cP \to \cN$ by
    \[
    \alpha (q) = 
    \begin{cases}
        \infty & (q=7)\\
        0 & (q\neq 7),
    \end{cases}
    \]
    the associated quotient consists of coproducts of loops whose periods are powers of $7$. This quotient is equivalent to the topos of continuous actions of a topological group $\Z_7$. This construction also works for any prime numbers other than $7$.
\end{example}

\begin{example}[Weird quotient]\label{ExampleWeird}
At last, we intentionally pick a weird quotient to observe the diversity of quotients.
    Defining $\alpha\colon \cP \to \cN$ by 
    \[\alpha(q) =2,\]
    a discrete dynamical system $\X = (X,f)$ belongs to the associated quotient, if and only if, for any element $x\in X$, $x$ goes into a loop within two steps, and its period has no cubic factors. 
\invmemo{Use this in Introduction}
\end{example}


\begin{remark}[Relation to topological monoid action toposes]
    In the above examples, we mentioned some topological monoids. As studied in \cite{rogers2023toposes}, (hyperconnected) quotients of monoid action toposes (including $\DD$) are closely related to topological monoid action toposes.
    
    The quotient of eventually periodic discrete dynamical systems (Example \ref{ExampleEventualyPeriodic}) is used to show that some hyperconnected quotients are not induced by a topology on the given monoid in \cite{rogers2023toposes}.
\end{remark}



\subsection{Quotient of eventual bijections}\label{subsectionEventualBijection}
Proposition \ref{PropositionConstructionPrimeQuotients} in the previous section is powerful and provides a method to construct $2^{\aleph_0}$ quotients. However, it does not cover all quotients. In this section, we will introduce another construction method. Later in Section \ref{SectionMainTheorem}, we will prove that these two methods together exhaust all possible quotients.

\begin{proposition}[Quotient of eventual bijections]\label{PropositionEventualBijConstruction}
    For a non-negative integer $a\in \N$, the full subcategory of $\DD$ that consists of a discrete dynamical system $\X=(X,f)$ that satisfies the following condition is a quotient.
    
    \textbf{Condition}: the restriction of $f$ to $\Image{f^{a}}$ 
        \[f
        \colon \mathrm{Im}(f^{a})\to \Image{f^{a}}
        \]
        is bijective.
\end{proposition}
\begin{proof}
    The proof will be given in Section \ref{SectionWhere}. 
\end{proof}

These quotients are not like those given by Proposition \ref{PropositionConstructionPrimeQuotients}. For example, they are not closed under taking subobjects, but closed under small limits.

\begin{example}[Quotient of bijections]
    In the case of $a=0$, the associated quotients consist of all bijections $f\colon X\to X$. The quotient is equivalent to the topos of $\Z$-actions. This quotient is similar to the quotient of loops (Example \ref{ExampleLoopsProfiniteCompletion}), but contains a non-loop bijection $(\Z, +1)$.

    This quotient is not closed under taking subobjects, since $(\Z,+1)$ belongs to it but $(\N,+1)$ does not.
\end{example}


In Section \ref{SectionMainTheorem}, we will prove the following corollary, which is one of several ways to state the conclusion of this paper.
\begin{corollary}[Classification of Quotients]\label{CorollaryElementaryInAdvance}
    Every class of discrete dynamical systems that is closed under taking finite limits and small colimits defines a quotient of $\DD$ and obtained by one of the following constructions:
    \begin{itemize}
        \item $\DD$ itself
        \item Quotient via Prime numbers (Proposition \ref{PropositionConstructionPrimeQuotients})
        \item Quotient of Eventual bijections (Proposition \ref{PropositionEventualBijConstruction}).
    \end{itemize}
\end{corollary}

\section{Generative order}\label{sectionGenerativeOrder}
In this section, we prepare a theoretical framework for classifying all quotients. The idea is simple. Roughly speaking, we will define a preorder on objects of $\E$,
which we will call the \emph{generative order}, 
that is equivalent to the poset of quotients of $\E$ and the inclusion relation (in the sense of Remark \ref{RemarkQuotientToposesAsClassesOfObjects}). 

To deal with some subtleties and make the idea into a rigorous argument, we introduce the notion of a \emph{prequotient}, which mediates those two order structures.


Throughout this section, we utilize some notions of order-preserving functions, especially \emph{embeddings} and \emph{dense functions}. For those who are not familiar with them, related definitions and properties are summarized in Appendix \ref{AppendixSectionDensity}.

\subsection{Prequotients and Generative order}
In this subsection, we will consider three (possibly large) posets:
\begin{description}
    \item[ $\Q_{\E}$ ]the poset of quotients and the inclusion order.
    \item[ $\PQ_{\E}$ ]the poset of \emph{prequotients} and the inclusion order.
    \item[ $\G_{\E}$ ]the poset of objects and the \emph{generative order}.
\end{description}
They have different definitions. But, in many cases, including $\E = \DD$, they all are small and isomorphic to each other. (The authors do not know any examples of Grothendieck toposes where these three posets are not isomorphic.)

Our idea is quite simple. The poset $\Q_{\E}$ is what we want to know, and $\G_{\E}$ is what we can calculate. The poset $\PQ_{\E}$ contains and mediates them, as shown in the following diagram.
\[
\begin{tikzcd}[column sep = 20pt]
    \Q_{\E}\ar[r,hookrightarrow
    ]&\PQ_{\E}\ar[r,hookleftarrow , "\iota_{\E}"
    ]&\G_{\E}
\end{tikzcd}
\]
The above two embeddings tend to be isomorphisms, as discussed in Remark \ref{RemarkQuotientandPrequotient} and Theorem \ref{TheoremFundLemma}.

We begin with the first poset $\Q_{\E}$, which is what we want to know, i.e., quotients.
\begin{notation}
    Let $\Q_{\E}$ denote the (possibly large) poset of all quotients of $\E$, equipped with the inclusion relation (see Remark \ref{RemarkQuotientToposesAsClassesOfObjects}).
\end{notation}
A priori, this poset $\Q_{\E}$ and the following two posets $\PQ_{\E}, \G_{\E}$ are not necessarily small.
However, for many Grothendieck toposes (including all Boolean Grothendieck toposes \cite{hora2023internal}), their sizes are proven to be small. The first problem of Lawvere's open problems in topos theory \cite{OpenLawvere} is asking whether $\Q_{\E}$ is small for an arbitrary Grothendieck topos $\E$.

In order to analyze $\Q_{\E}$, we introduce the second poset $\PQ_{\E}$ of \emph{prequotients}, which contains $\Q_{\E}$.
\begin{definition}[Prequotient]\label{DefinitionPrequotients}
    A \emph{prequotient} of a cocomplete topos $\E$ is a full subcategory of $\E$, which is closed under 
    \begin{itemize}
        \item taking isomorphic objects,
        \item finite limits, and
        \item small colimits.
    \end{itemize}
\end{definition}
As mentioned in Remark \ref{RemarkQuotientToposesAsClassesOfObjects}, we sometimes treat a prequotient as the
corresponding class of objects.

\begin{notation}
    $\PQ_{\E}$ denotes the (possibly large) poset of prequotients of a cocomplete topos $\E$ equipped with the inclusion relation as a partial order.
\end{notation}

Identifying a quotient with the essential image of the inclusion functor, we obtain the following lemma, which states $\Q_{\E}\subset \PQ_{\E}$. This identification does not cause any problem because we do not distinguish two quotients that contain exactly the same isomorphism classes, as explained in Section \ref{subsectionDefQuotients}.
\begin{lemma}\label{LemmaQuotientIsPrequotient}
A quotient $Q$ of a cocomplete topos $\E$ is a prequotient. Therefore, $\Q_{\E}$ is a subposet of $\PQ_{\E}$.
\end{lemma}
\begin{proof}
   This follows from the fact that a quotient is a coreflective subcategory whose inclusion functor preserves finite limits \cite[see][Proposition 4.5.15 or Theorem 5.6.5]{riehl2017category}.
\end{proof}

\begin{remark}[Relation between quotients and prequotients]\label{RemarkQuotientandPrequotient}
    After proving a quotient is a prequotient, it is natural to ask whether the converse statement holds.
    Roughly speaking, the answer is \dq{almost YES.} In fact, at least under V\v{o}penka principle, every prequotient of a Grothendieck topos is a quotient \cite[see][]{adamek1994locally}. For the topos $\DD$, we will observe that they coincide without any additional assumptions (see Corollary \ref{CorollaryDDprequotientIsQuotient}). But for a general case, the authors do not know the answer.
\end{remark}

For what do we consider a prequotient, instead of a quotient? The short answer is that prequotients are easier to deal with than quotients. The definition of prequotient does not involve the existence of an adjoint functor. 
All required conditions are of the form of \dq{being closed by some operations}.

Thanks to this simple definition, 
we can easily prove that, for each object $X\in \E$, there exists the smallest prequotient that contains $X$ (the next lemma).
This property leads us to define the third poset $\G_{\E}$, the \emph{generative order} poset.

\begin{lemma}
    For a cocomplete topos $\E$ and a subclass $A\subset \ob{\E}$, there exists the smallest prequotient that contains $A$.
\end{lemma}
\invmemo{Proof?}

\begin{notation}
    For any subclass $A\subset \ob{\E}$, let $\gen{A}$ denote the minimum prequotient that contains $A$. 
We abuse this notation and let $\gen{X}$ denotes $\gen{\{X\}}$ for an object $X\in \ob{\E}$.
\end{notation}



\begin{definition}[Generative order]
    For two objects $X, Y$ of a cocomplete topos $\E$, the \textit{generative order} $X \ogeq Y$ is defined to be $\gen{X}\subset \gen{Y}$. As usual, $X\og Y$ means $X \ogeq Y$ and $X \not\goeq Y$. $X \gog Y$ means $X\ogeq Y$ and $X \goeq Y$.
\end{definition}

This relation defines a preorder on the class of all objects of $\E$.

\begin{notation}
        $\G_{\E}$ denotes the (possibly large) poset obtained by quotienting $(\ob{\E}, \ogeq)$ by $\gog$.
\end{notation}

By abuse of notation, when it does not cause confusion, the equivalence class $[X]$ of an object $X$ will be simply denoted by $X$.
Some examples of concrete calculations of $\G_{\E}$ will be given in this section (Example \ref{ExampleGSet}, Example \ref{ExampleGFunction}, and Example \ref{ExampleZAction}).

\begin{example}[Minimum elements]\label{ExampleMinimumElement}
The terminal object $\1$ and the initial object $\0$ of a cocomplete topos give the minimum element $[\1]=[\0] $ of $ \G_{\E}$,
because every prequotient contains them. There are many other representatives of the minimum elements. For example, $[\1 + \1]$ is also the minimum element.
\end{example}

For a subobject $S\rightarrowtail X$, the inequality $S \ogeq X$ may not hold.
However, if $S$ is a member of two special classes of subobjects, namely direct summands and retracts, then it holds.

\begin{lemma}[Direct summands]\label{LemmaCoproduct}
 For two objects $X, Y$, their coproduct $X+Y$ is greater or equivalent to $X$ and $Y$ with respect to the generative order.
 \[X+Y \goeq X\]
\end{lemma}
\begin{proof}
Take an arbitrary prequotient $Q$ that contains $X+Y$. 
Since there is a pullback diagram
\[
\begin{tikzcd}
X\ar[r]\ar[d]  \arrow[dr, phantom, "\lrcorner", very near start]&\1\ar[d]\\
X+Y\ar[r]&\1 +\1,
\end{tikzcd}
\]
$Q$ contains $X$. Here, we used the fact that any topos is an extensive category.
\end{proof}

\begin{lemma}[Retracts]\label{LemmaRetract}
If an object $R$ is a retract of another object $X$, then $X \goeq R$.
\end{lemma}
\begin{proof}
Retract is the equalizer of the associated idempotent morphism $e$ on $X$ and $\id_{R}$
\[
\begin{tikzcd}
    R\ar[r,rightarrowtail]&X\ar[r,"e",shift left]\ar[r,"\id_{X}"',shift right]&X.
\end{tikzcd}
\]
\end{proof}


By the definition, there is a natural order-preserving function
\[\iota_{\E}\colon \G_{\E}\hookrightarrow \PQ_{\E}\colon X \mapsto \gen{X}.\]
The next lemma is easy but fundamental.
\begin{lemma}
\label{LemmaDenseEmbeddingIota}
For a cocomplete topos $\E$, $\iota_{\E}\colon \G_{\E}\hookrightarrow \PQ_{\E}$ is a dense embedding (see Appendix \ref{AppendixSectionDensity}).
\end{lemma}
\begin{proof}
The related definitions immediately imply that $\iota_{\E}$ is an embedding. We prove $\iota_{\E}$ is dense.
The function in Definition \ref{DefinitionDense}
$y^{\iota_{\E}}\colon \PQ_{\E}\to \Dw{\G_{\E}}$, where $\Dw{\G_{\E}}$ denotes the poset of the downward closed subsets,
sends a prequotient $Q$ to \[y^{\iota_{\E}} (Q) = \{X\in \G_{\E} \mid \gen{X} \subset Q\} = \{X\in \G_{\E} \mid X \in Q\} = Q/\mathord{\gog} \subset \G_{\E}.\] 
Because a prequotient is a union of equivalence classes of $\mathord{\gog}$, this proves $y^{\iota_{\E}}$ is an embedding and $\iota_{\E}$ is dense.

\end{proof}

\begin{proposition}
    For a cocomplete topos $\E$, $\PQ_{\E}$ is small if and only if $\G_{\E}$ is small.
\end{proposition}
\begin{proof}
    This immediately follows from Lemma \ref{LemmaDenseEmbeddingIota} and Proposition \ref{PropositionSmallnessComparison}.
\end{proof}
Since Lawvere's open problem is asking about the smallness of the class of quotients $\Q_{\E}\subset \PQ_{\E}$, it is enough to prove that $\G_{\E}$ is small.

The next proposition states that this function $\iota_{\E}$ is not only dense but cocontinuous (i.e., preserves all small supremums).
\begin{proposition}\label{PropositionFullCoprod}
    Let $\E$ be a cocomplete topos and $\{X_{\lambda}\}_{\lambda \in \Lambda}$ be a small family of objects.
    \begin{itemize}
        \item In $\PQ_{\E}$, the supremum of $\{\gen{X_{\lambda}}\}_{\lambda \in \Lambda}$ is given by $\gen{\{X_{\lambda}\}_{\lambda \in \Lambda}} = \gen{\coprod_{\lambda \in \Lambda} X_{\lambda}} $.
        \item In $\G_{\E}$, the supremum  of $\{X_{\lambda}\}_{\lambda \in \Lambda}$ is given by $\coprod_{\lambda \in \Lambda} X_{\lambda}$.
    \end{itemize}
\end{proposition}
\begin{proof}
    By the related definitions, it is easy to prove that $\gen{\{X_{\lambda}\}_{\lambda \in \Lambda}}$ is the supremum of $\{\gen{X_{\lambda}}\}_{\lambda \in \Lambda}$. We prove the equality $\gen{\{X_{\lambda}\}_{\lambda \in \Lambda}} = \gen{\coprod_{\lambda \in \Lambda} X_{\lambda}} $. Since a prequotient is closed under taking small corpuducts, we have \[\gen{\coprod_{\lambda \in \Lambda} X_{\lambda}} \subset 
        \gen{\{X_{\lambda}\}_{\lambda \in \Lambda}}.\] For the opposite direction, we can utilize Lemma \ref{LemmaCoproduct} and the equation 
        \[
        \coprod_{\lambda} X_{\lambda} = X_{\lambda_0} + \coprod_{\lambda \neq \lambda_0} X_{\lambda}.
        \]

        The latter statement follows from the former statement and the fact that $\iota_{\E}$ is an embedding (see Lemma \ref{LemmaEmbeddingSupremum}).
\end{proof}
\invmemo{Embedding and order-embedding}

In category-theoretic terms, the above proposition states that the possibly large poset $\G_{\E}$ is (small) cocomplete, the colimit (supremum) is given by the coproduct, and $\iota_{\E}$ is cocontinuous. 

\begin{example}[Maximum elements]\label{ExampleMaximumElement}
For a Grothendieck topos $\E$, there is the maximum element of $\G_{\E}$. Take a small site $(\C, J)$. Then, the coproduct of the sheafifications of representable presheaves
\[
\coprod_{c\in \ob{\C}} \mathbf{ay}(c)
\]
is (a representative of) the maximum element, since every $J$-sheaf is a colimit of a diagram consisting of $\mathbf{ay}(c)$. For example,
$\Nsucc$ is the maximum element of $\G_{\DD}$.
\end{example}

Furthermore, for many cases, there is a much more direct relationship between $\G_{\E}$ and $\PQ_{\E}$: they are isomorphic! We call it the fundamental lemma of generative order.
\begin{theorem}[Fundamental Lemma]\label{TheoremFundLemma}
If $\G_{\E}$ (or equivalently $\PQ_{\E}$) is small, then $\iota_{\E}\colon \G_{\E}\to \PQ_{\E}$ is an order isomorphim.
\end{theorem}
\begin{proof}
    Since we have proven that $\iota_{\E}$ is an embedding (Lemma \ref{LemmaDenseEmbeddingIota}), it is enough to prove that $\iota_{\E}$ is surjective. Take an arbitrary prequotient $Q \in \PQ_{\E}$. Then $Q/\mathord{\gog} (= y^{\iota_{\E}}(Q))$ is a downward closed subset of $\G_{\E}$. 
    Furthermore, it is closed by taking arbitrary small supremums by Proposition \ref{PropositionFullCoprod}, 
    since a prequotient is closed under taking small coproducts.

    Therefore, since $Q/\mathord{\gog}  \subset \G_{\E}$ is small and closed under taking small supremums, it has the maximum element $X$ and is generated by that object, i.e., $Q = \gen{X}$. Then $Q=\iota_{\E} (X)$ and the proof is completed.
\end{proof}



\begin{example}\label{ExampleGSet}
    As the simplest examples for classifying all quotients, we rephrase Example \ref{ExampleQuotientSets} in terms of the generative order.
    A singleton is the minimum and the maximum element in $\G_{\Set}$ (see Example \ref{ExampleMinimumElement} and Example \ref{ExampleMaximumElement}).
    Therefore $\G_{\Set}$ is a one-element poset and in particular small. Therefore, by Theorem \ref{TheoremFundLemma}, $\PQ_{\Set}$ is also a one-element poset and the only prequotient of $\Set$ is $\Set$ itself. 
\end{example}



\subsection{For locally connected toposes}\label{subsectionLocallyConnectedTopos}
In the last section, we have observed that $\Q_{\E}$ can be analyzed via $\G_{\E}$. In this subsection, we present a way to analyze $\G_{\E}$, assuming $\E$ is a \emph{locally connected topos}.

First, we recall the definition of connected objects and locally connected topos.

\begin{definition}
    An object $X$ of a Grothendieck topos $\E$ is \textit{connected} if it satisfies the following equivalent conditions:
    \begin{itemize}
        \item The hom functor $\E(X,-)\colon \E \to \Set$ preserves small coproducts.
        \item The hom functor $\E(X,-)\colon \E \to \Set$ preserves finite coproducts.
        \item $X$ is not an initial object, and if $X$ is a coproduct of two objects $X=S \coprod T$, then either $S$ or $T$ is an initial object.
    \end{itemize}
\end{definition}
\invmemo{Proof of the equivalence}

In the case where $\E = \DD$, this definition is equivalent to the connectedness of a discrete dynamical system (see Definition \ref{DefinitionConnectedDD}).


\begin{definition}
    A Grothendieck topos is called \textit{locally connected} if every object is a coproduct of connected objects.
\end{definition}

\begin{example}[Presheaf category]
    Every presheaf category over a small category is a locally connected Grothendeick topos. In particular, $\DD$ is locally connected.
\end{example}

\begin{example}[Sheaf over a topological space]
    For a topological space $X$, its sheaf topos $\mathrm{Sh}(X)$ is locally connected if and only if $X$ is locally connected as a topological space.
\end{example}
\begin{notation}
    For a cocomplete topos $\E$, $\Gc_{\E}$ denotes the subposet of $\G_{\E}$ that consists of connected objects.
\end{notation}
In more detail, an element of $\Gc_{\E}$ is an equivalence class of $\ob{\E}$ that contains at least one connected object.
    

\begin{lemma}\label{LemmaConnectedDense}
    For a locally connected Grothendieck topos $\E$,
    the embedding 
    $\Gc_{\E} \to \G_{\E}$ is dense. 
\end{lemma}
\begin{proof}
    According to Definition \ref{DefinitionDense}, what we need to prove is that the function
    \[
    y^{\iota}\colon \G_{\E}\to \Dw{\Gc_{\E}}
    \]
    induced by the embedding $\iota\colon \Gc_{\E}\to \G_{\E}$ is an embedding.
    Since it is automatically order-preserving, it is enough to prove \[y^{\iota}(X)\subset y^{\iota}(Y)
    \implies X\ogeq Y
    .\]
    
    Since $\E$ is locally connected, $X$ can be decomposed as
    \[X \cong \coprod_{\lambda\in \Lambda} X_{\lambda}\]
    with a small family of connected objects $\{X_{\lambda}\}_{\lambda\in \Lambda}$. 
    By Lemma \ref{LemmaCoproduct}, we have $\{X_\lambda\}_{\lambda\in \Lambda} \subset y^{\iota}(X)\subset y^{\iota}(Y)$, i.e., $X_\lambda \ogeq Y$ for any $\lambda \in \Lambda$.
    By Proposition \ref{PropositionFullCoprod}, which states $\coprod_{\lambda\in \Lambda} X_{\lambda}$ is the supremum of $\{X_\lambda\}_{\lambda\in \Lambda}$ in $\G_{\E}$, we obtain $X\ogeq Y$.
\end{proof}

Finally, combining everything in this section, we can embed $\Q_{\E}$ into $\Dw{\Gc_{\E}}$. It means that classification of quotients is reduced to the calculation of connected objects.

\begin{proposition}\label{PropositionChainOfEmbeddings}
For a locally connected Grothendieck topos $\E$,
    if $\Gc_{\E}$ is small, we have a chain of order-embeddings
\[\Q_{\E}\hookrightarrow\PQ_{\E}\cong \G_{\E} \hookrightarrow \Dw{\Gc_{\E}}.\]
\end{proposition}
\begin{proof}
    The right-most embedding is due to Lemma \ref{LemmaConnectedDense}. By Lemma \ref{PropositionSmallnessComparison}, $\G_{\E}$ is small. Then, we can apply Theorem \ref{TheoremFundLemma} to obtain the middle isomorphism.
\end{proof}

Although the explanation has been long, the above embedding just sends a quotient $Q$ to the set of all connected objects that $Q$ contains.


\begin{example}[Topos of functions]\label{ExampleGFunction}
    As a somewhat non-trivial example, we classify all quotients of the topos of functions $\E \coloneqq \Set^{\to}$. A function $f\colon A\to B$ (which is an object of this category) is connected if and only if the codomain $B$ is a singleton. For each cardinal $\kappa$, let $f_\kappa$ denote the unique function $\kappa\to \{\ast\}$. Since $f_1$ is the terminal object, it is the minimum element in $\G_{\E}$ (Example \ref{ExampleMinimumElement}) and so is in $\Gc_{\E}$.

    We prove that all other connected objects are equivalent to each other, with respect to $\ogeq$.
    First, $f_0$ is the maximum element, because $f_0 \goeq \sup\{f_0, f_1\}=f_0 \coprod f_1$ (Proposition \ref{PropositionFullCoprod}), and $f_0 \coprod f_1$ is the maximum element (Example \ref{ExampleMaximumElement}). 
    For $\kappa \geq 2$, by considering an endofunction $g:\kappa\to\kappa$ without a fixed point,
    and equalizer of $g$ (at the domain part) and the identity, 
    we obtain $f_{\kappa} \goeq f_{0}$, hence $f_{\kappa} \gog f_{0}$. 
    
    Now we have proven that $\Gc_{\E}$ has at most two elements. Since the full subcategory of $\E$ that consists of all bijections is a quotient containing $f_1$ but not $f_0$, we have $f_0 \go f_1$. Now we have proven that $\Gc_{\E}$ is the $2$-element totally ordered set.

    Therefore, by Proposition \ref{PropositionChainOfEmbeddings}, $\PQ_{\E}\cong \G_{\E}$ is embeddable into the $3$-elements totally ordered set $\Dw{\Gc_{\E}}$. Since quotients are non-empty, we conclude that $\E$ itself and the quotient consisting of all bijections are the only quotients of the topos of functions $\E$.
    (The quotient of all bijections is equivalent to the topos of sets $\Set$, and this quotient witnesses the connectedness of the topos of functions due to Example \ref{ExampleQuotientAsConnectedness}.)
\end{example}

As mentioned in Example \ref{ExampleTopologicalGroupQuotient}, every quotient of a group action topos is known to be induced by a topological group structure on the group. However, it is worth explaining the classification of quotients of $\ps{\Z}$ in our framework,
because it is similar to the classification of quotients of $\DD$. For the notation $\cpN$, see Notation \ref{NotationNNN}.
\begin{example}[Topos of $\Z$-sets]\label{ExampleZAction}
    We classify all quotients of the topos of $\Z$-sets $\E \coloneqq \ps{\Z}\cong \Set^{\Z}$, by constructing an order isomorphism
    \[
    \Gc_{\E} \cong \cpN.
    \]
    Since every connected object is isomorphic to $C_n \coloneqq (\Z/n\Z, +1)$ for some $n\in \N$, it is enough to prove 
    \[n \text{ divides } m \iff C_n \ogeq C_m. \]
    Here, $C_0$ means $(\Z,+1)$. If $n$ divides $m$, then $C_n$ is a coequalizer of 
    \[
    \begin{tikzcd}
        C_m \ar[r, shift right," +0"'] \ar[r, shift left,"+n "] &C_m,
    \end{tikzcd}
    \]
    which implies $C_n \ogeq C_m$. 
    On the other hand, if $n$ does not divide $m$, then the quotient $\E \hookleftarrow \ps{\Z/m\Z}$, induced by the surjective group homomorphism $\Z \twoheadrightarrow \Z/m\Z$ (Example \ref{ExampleBOFquotient}), contains $C_m$ but not $C_n$. This implies $C_n \not \ogeq C_m$. Thus we have proven $\Gc_{\E} \cong \cpN$.

    By Proposition \ref{PropositionChainOfEmbeddings}, we obtain an embedding $\Q_{\E} \hookrightarrow \Dw{\cpN}$. The remaining task is determining the image of the embedding. For $n,m \in \cpN$, we can prove
    \[
    C_n \times C_m \cong \underbrace{C_{\mathrm{lcm}(n,m)} + \dots +C_{\mathrm{lcm}(n,m)}}_{\mathrm{gcd}(n,m) \text{ times}},
    \]
    except that $C_0 \times C_0$ is a coproduct of countably many $C_0$. 
    This implies that a realizable $D\subset \cpN$ should be closed under taking finite supremums, i.e., $D$ should be an \emph{ideal} of $\cpN$. 
    Conversely, for any ideal $I \subset \cpN$, by considering a topological group structure on $\Z$ where the family of subgroups $\{n\Z (\subset \Z)\}_{n\in I}$ forms a fundamental neighborhood system of $0 \in \Z$, one can construct the corresponding quotient topos (see Example \ref{ExampleTopologicalGroupQuotient} and \cite{hora2023internal}).
    Then, we have obtained an order isomorphism
    \[\Q_{\E} \cong \mathrm{Ideal}(\cpN).\]
    
\end{example}

\section{Quotients of the topos of discrete dynamical systems}\label{SectionMainTheorem}
Now, it is time to classify the quotients of $\DD$!
Based on Proposition \ref{PropositionChainOfEmbeddings}, our method is to determine the structure of $\Gc_{\DD}$.

\subsection{Constructions of representatives}
We will construct a poset isomorphism
\[\Gc_{\DD} \cong \NN,\]
where
$\NN$ is equipped with the product order (see Notation \ref{NotationNNN}). This is similar to the case of $\ps{\Z}$ (Example \ref{ExampleZAction}), but $\DD$ is not so straightforward.

To construct an isomorphism $\NN \to \Gc_{\DD}$, we define a connected discrete dynamical system $\T{a}{b} \in \Gc_{\DD}$ for each $(a,b) \in \NN$. Informally, $\T{a}{b}$ looks like a loop $(\Z/ b\Z, +1)$ with a \dq{tail} of length $a$. 
Before looking at the definition below, it might be quicker to see the conclusion (Figure \ref{FigureTs}).
\begin{figure}[ht]
    \begin{shaded}
        \centering
        \begin{tikzpicture} [scale = \scvalue]
        \def \r{2/pi}
        \def \g{-2.3}
        
        \def \a{0}
        \node at (-5, \a -0.5) {$\T{3}{4}\colon$};
        \foreach \x in {-3, -2, -1} {
            \draw[black, thick, \arst] (\x, \a) -- (\x + 1, \a);
            \fill[black] (\x,\a) circle (0.06);
        }
        \draw[black, thick, \arst] (0,\a)     arc (90  :  0:\r);
        \draw[black, thick, \arst] (\r,\a-\r) arc (360 :270:\r);
        \draw[black, thick, \arst](0, \a-2*\r)arc (270 :180:\r);
        \draw[black, thick, \arst] (-\r,\a-\r)arc (180 : 90:\r);
        \fill[black] (0,\a) circle (0.06);
        \fill[black] (\r,\a-\r) circle (0.06);
        \fill[black] (0, \a-2*\r) circle (0.06);
        \fill[black] (-\r,\a-\r) circle (0.06);

        \def \b{\g}
        \node at (-5, \b -0.5) {$\T{\infty}{4}\colon$};
        \foreach \x in {-3, -2, -1} {
            \draw[black, thick, \arst] (\x, \b) -- (\x + 1, \b);
            \fill[black] (\x,\b) circle (0.06);
        }
        \draw[black, thick, \arst] (-4+0.5, \b) -- (-4 + 1, \b);
        \node at (-4, \b) {$\cdots$};
        \draw[black, thick, \arst] (0,\b)     arc (90  :  0:\r);
        \draw[black, thick, \arst] (\r,\b-\r) arc (360 :270:\r);
        \draw[black, thick, \arst](0, \b-2*\r)arc (270 :180:\r);
        \draw[black, thick, \arst] (-\r,\b-\r)arc (180 : 90:\r);
        \fill[black] (0,\b) circle (0.06);
        \fill[black] (\r,\b-\r) circle (0.06);
        \fill[black] (0, \b-2*\r) circle (0.06);
        \fill[black] (-\r,\b-\r) circle (0.06);

        \def \c{2 * \g}
        \node at (-5, \c -0.5) {$\T{3}{0}\colon$};
        \foreach \x in {-3, -2} {
            \draw[black, thick, \arst] (\x, \c) -- (\x + 1, \c);
        }
        \foreach \x in {-3, -2, -1} {
            \fill[black] (\x,\c) circle (0.06);
        }
        \draw[black, thick, \arst] (-1, \c) -- (0, \c -1);
        \foreach \x in {-3, -2, -1, 0, 1} {
            \draw[black, thick, \arst] (\x, \c -1) -- (\x + 1, \c -1);
        }
        \draw[black, thick, \arst] (-3.5, \c -1) -- (-3, \c -1);
        \node at (-4, \c -1) {$\cdots$};
        \draw[black, thick] (2, \c -1) -- (2.5, \c -1);
        \node at (3, \c -1) {$\cdots$};
        \foreach \x in {-3, -2, -1, 0, 1, 2} {
            \fill[black] (\x,\c -1) circle (0.06);
        }

        \def \d{3 * \g}
        \node at (-5, \d -0.5) {$\T{\infty}{0}\colon$};
        \foreach \x in {-3, -2} {
            \draw[black, thick, \arst] (\x, \d) -- (\x + 1, \d);
        }
        \draw[black, thick, \arst] (-3.5, \d) -- (-3, \d);
        \node at (-4, \d) {$\cdots$};
        \foreach \x in {-3, -2, -1} {
            \fill[black] (\x,\d) circle (0.06);
        }
        \draw[black, thick, \arst] (-1, \d) -- (0, \d -1);
        \foreach \x in {-3, -2, -1, 0, 1} {
            \draw[black, thick, \arst] (\x, \d -1) -- (\x + 1, \d -1);
        }
        \draw[black, thick, \arst] (-3.5, \d -1) -- (-3, \d -1);
        \node at (-4, \d -1) {$\cdots$};
        \draw[black, thick] (2, \d -1) -- (2.5, \d -1);
        \node at (3, \d -1) {$\cdots$};
        \foreach \x in {-3, -2, -1, 0, 1, 2} {
            \fill[black] (\x,\d -1) circle (0.06);
        }
        
        \end{tikzpicture}
        
        \caption{Visualization of $\T{a}{b}$ for some $(a,b)$}
        \label{FigureTs}
    \end{shaded}
\end{figure}


\begin{definition}[$\T{a}{b}$]
    For $(a,b)\in \NN$, we define a connected discrete dynamical system $\T{a}{b}$ as follows:
    \begin{itemize}
        \item The underlying set is a disjoint union of 
        $\{i\in \Z\mid -a\leq i < 0 \}$ and $\Z/b\Z$.
        \item The associated endofunction sends $\overline{j}\in \Z/b\Z$ to $\overline{j+1} \in \Z/b\Z$ and $-a\leq i < 0$ to
        \[
        \begin{cases}
            i+1  & (-a\leq i < -1)\\
            \overline{0} \in \Z/b\Z & (i=-1).
        \end{cases}
        \]
    \end{itemize}
\end{definition}
In subsection \ref{SubsectionConnectedIsom}, we will prove this correspondence $\NN \to \Gc_{\DD}$ is an isomorphism. 

Before proceeding,
we give a simple observation on $\T{\infty}{0}$. As you can see in Figure \ref{FigureTs}, $\T{\infty}{0}$ has a non-trivial automorphism, which flips up and down. Utilizing this fact, we can prove the following lemma.
\begin{lemma}\label{LemmaNsuccisMaximum}
    $\T{\infty}{0}$ is  equivalent to $\Nsucc$ with respect to the generative order.
    \[\T{\infty}{0} \gog \Nsucc\]
    Furthermore, it is the maximum element of $\G_{\DD}$.
\end{lemma}
\begin{proof}
    Since $\Nsucc$ is the maximum element of $\G_{\DD}$ (Example \ref{ExampleMaximumElement}), it suffices to prove $\T{\infty}{0} \goeq \Nsucc$. 
    To prove that, take the equalizer of the automorphism that flips up and down in Figure \ref{FigureTs} and the identity.
\end{proof}


\subsection{Generalized period and height}
In Definition \ref{DefinitionPeriodHeightElement}, we have defined the period and height of an element of a discrete dynamical system.
In this subsection, we define them for a connected discrete dynamical system, by extending those notions for an element. Visualized examples (Figure \ref{FigureHeightExample}) will be helpful for quickly becoming familiar with those notions.
\begin{figure}[ht]
    \begin{shaded}
        \centering
        \begin{tikzpicture} [scale = \scvalue]
        \def \r{2/pi}
        \def \g{-2.3}
        \def \core{0.06}
         \def \b{-4}

        \foreach \x in {3, 2, 1} {
            \draw[gray, thick, \arst] (-\x, \b) -- (-\x + 1, \b);
            \fill[gray] (-\x,\b) circle (0.06);
            \node at (-\x, \b)[above] {${\x}$};
        }
        \foreach \x in {3, 2, 1, 0} {
            \node at (-\x, \b)[above] {${\x}$};
        }

        \foreach \x in {1,2,3,4} {
            \draw[gray, thick, \arst] (\x, \b - 2*\r) -- (\x - 1, \b - 2*\r);
            \fill[gray] (\x, \b - 2*\r) circle (0.06);
        }

        \foreach \x in {0,1,2,3,4} {
            \node at (\x, \b - 2*\r)[below] {${\x}$};
        }

        \foreach \x in {1} {
            \draw[gray, thick, \arst] (\r,\b+ \x- \r) -- (\r, \b+\x- \r -1);
            \fill[gray] (\r,\b+ \x- \r) circle (0.06);
        }

        \foreach \i in {1,2,3,4, 5} {
        \pgfmathsetmacro{\x}{\r+ cos(90 / \i)}
        \pgfmathsetmacro{\y}{\b +1 - \r + sin(90/ \i)}
            \draw[gray, thick, \arst] (\x, \y) -- (\r, \b+1- \r);
            \fill[gray] (\x, \y) circle (0.06);
            \node at (\x, \y)[right] {$2$};
        }

        \foreach \i in {-1.5,-0.5,0.5} {
        \pgfmathsetmacro{\x}{\r+ cos( - 5 * \i )}
        \pgfmathsetmacro{\y}{\b +1 - \r + sin(- 5 * \i)}
            \node at (\x, \y) {$\cdot$};
        }

        \draw[gray, thick, \arst] (-1, \b + 0.5) -- (0,\b);
        \fill[gray] (-1, \b + 0.5) circle (0.06);
        \node at (-1, \b + 0.5)[above] {$1$};

        \draw[gray, thick, \arst] (-3, \b - 0.5) -- (-2,\b);
        \fill[gray] (-3, \b - 0.5) circle (0.06);
        \node at (-3, \b - 0.5)[above] {$3$};

        \draw[gray, thick, \arst] (3, \b - 2 * \r + 0.5) -- (2, \b - 2 * \r);
        \fill[gray] (3, \b - 2 * \r + 0.5) circle (0.06);
        \node at (3, \b - 2 * \r + 0.5)[below] {$3$};

        \foreach \x in {4,5}{
        \draw[gray, thick, \arst] (\x, \b - 2 * \r + 0.5) -- (\x -1, \b - 2 * \r + 0.5);
        \fill[gray] (\x, \b - 2 * \r + 0.5) circle (0.06);
        \node at (\x, \b - 2 * \r + 0.5)[below] {$\x$};
        }

        \draw[gray, thick, \arst] (5, \b - 2 * \r + 1) -- (4, \b - 2 * \r + 0.5);
        \fill[gray] (5, \b - 2 * \r + 1) circle (0.06);
        \node at (5, \b - 2 * \r + 1)[below] {$5$};

        \draw[black, very thick, \arst] (0,\b)     arc (90  :  0:\r);
        \draw[black, very thick, \arst] (\r,\b-\r) arc (360 :270:\r);
        \draw[black, very thick, \arst](0, \b-2*\r)arc (270 :180:\r);
        \draw[black, very thick, \arst] (-\r,\b-\r)arc (180 : 90:\r);
        \fill[black] (0,\b) circle (\core);
        \fill[black] (\r,\b-\r) circle (\core);
        \fill[black] (0, \b-2*\r) circle (\core);
        \fill[black] (-\r,\b-\r) circle (\core);

        \foreach \x in {0,1,2} {
            \node at (\r,\b+ \x- \r)[right] {${\x}$};
        }
        \node at (-\r,\b-\r)[left] {$0$};

        \node at (1,\b -2) {(height, period) $=(5,4)$};

        \def \c{-10.5}
        


        \foreach \x in {-1,1,3} {
            \foreach \y in {1,2} {
                \draw[gray, thick, \arst] (\x, \c+\y)-- (\x, \c+\y-1);
                \fill[gray] (\x, \c+\y) circle (0.06);
                \node at (\x, \c+\y)[right] {$\y$};
            }
        }

        \foreach \x in {-2,0,2,4} {
            \foreach \y in {1,2,3} {
                \draw[gray, thick, \arst] (\x, \c+\y)-- (\x, \c+\y-1);
                \fill[gray] (\x, \c+\y) circle (0.06);
                \node at (\x, \c+\y)[right] {$\y$};
            }
        }

        \foreach \x in {-2,-1,0,1,2,3}{
        \draw[black, very thick, \arst] (\x, \c) -- (\x + 1, \c);
        }
        \draw[black, very thick, \arst] (-2.5, \c) -- (-2, \c);
        \draw[black, very thick] (4, \c) -- (4.5, \c);
        \node at (-3, \c) {$\cdots$};
        \node at (5, \c) {$\cdots$};
        \foreach \x in {-2,-1,0,1,2,3,4}{
        \fill[black] (\x,\c) circle (\core);
        \node at (\x,\c)[below] {$0$};
        }
        \node at (1,\c -0.7) {(height, period) $=(3,0)$};

        \def \d{-12.5}

        

        \foreach \x in {-2, -1} {
            \draw[gray, thick, \arst] (\x, \d) -- (\x + 1, \d);
        }
        \draw[gray, thick, \arst] (-2.5, \d) -- (-2, \d);
        \node at (-3, \d) {$\cdots$};
        \foreach \x in {-2, -1, 0} {
            \fill[gray] (\x,\d) circle (0.06);
            \node at (\x, \d)[below] {$\infty$};
        }
        \draw[gray, thick, \arst] (0, \d) -- (1, \d -1);
        
        \foreach \x in {-2, -1, 0, 1,2,3} {
            \draw[gray, thick, \arst] (\x, \d -1) -- (\x + 1, \d -1);
        }
        \draw[gray, thick, \arst] (-2.5, \d -1) -- (-2, \d -1);
        \node at (-3, \d -1) {$\cdots$};
        \draw[gray, thick] (4, \d -1) -- (4.5, \d -1);
        \node at (5, \d -1) {$\cdots$};
        \foreach \x in {-2, -1, 0, 1, 2, 3, 4} {
            \fill[gray] (\x,\d -1) circle (0.06);
            \node at (\x, \d -1)[below] {$\infty$};
        }

        \node at (1,\d -1.7) {(height, period) $=(\infty, 0)$};
        \end{tikzpicture}
        
        \caption{Visualized examples of period, height, and core}
        \label{FigureHeightExample}
    \end{shaded}
\end{figure}

What roles do period and height play in our context of analyzing $\Gc_{\DD}$?
The answer is quite simple: we will prove that the correspondence
\[
\Gc_{\DD} \ni X\mapsto \text{(height of $X$, period of $X$)} \in \NN
\]
provides the inverse function to the one defined in the previous section.


Defining period is relatively easy, due to the next lemma.
\begin{lemma}
    For a connected discrete dynamical system, all elements have the same period.
\end{lemma}
\begin{proof}
    It is because of Condition \ref{ConditionConcreteConnectedness} in Definition \ref{DefinitionConnectedDD}.
\end{proof}

This lemma verifies the next definition. (Notice that we also use the non-emptyness of a connected discrete dynamical system.)
\begin{definition}[Period]
    The \emph{period} of a connected discrete dynamical system
    is the period of any of its elements.
\end{definition}

Next, we define the height of a connected discrete dynamical system. Roughly speaking, we define it as a supremum of the height of all elements. However, the problem is that, the height is not defined for elements with the period of $0$. To deal with this problem, we first extend the definition of height to all elements.


One might think, \dq{Well, if elements with a period of $0$ never enter a loop, then we could just set their height to infinity.} However, this approach does not work well for later discussions. In fact, by slightly relaxing the concept of a loop, we can define a more appropriate quantity. We call this relaxed concept a \emph{core}. In Figure \ref{FigureHeightExample}, the core and the other part are colored by black and gray, respectively.

\begin{definition}[Core]
    For a connected discrete dynamical system $\X$ with the period $b\in \cpN$, if there is a unique subobject
    that is isomorphic to $\T{0}{b}\cong(\Z/b\Z, +1)$, then we call it the \emph{core} of $\X$.
\end{definition}

\begin{example}[The case where $b>0$]
    If the period of a connected discrete dynamical system $\X = (X,f)$ is positive, then $\X$ always has a core, which is the set of periodic elements
    \[
    \{x\in X \mid \text{there exists }n>0\text{ such that } f^{n}(x)=x\}.
    \]
\end{example}

\begin{example}[The case where $b=0$]
    $(\Z,+1)$ has a core, but
    $\Nsucc$ and $\T{\infty}{0}$ do not have a core.
    \invmemo{Write more}
\end{example}

\begin{definition}[Height]\label{DefinitiongeneralizedHeight}
    Let $\X= (X,f)$ be a connected discrete dynamical system.
    \begin{itemize}
        \item If $X$ has a core $C\subset \X$, then the height of an element $x\in X$ is the minimum natural number $n \in \N$ such that $f^n (x) \in C$. Otherwise, the height of every element is $\infty$.
        \item The height of $\X$ is the supremum of the heights of its elements in the poset $\cN$.

    \end{itemize}
\end{definition}

Due to the connectedness, if $\X$ has a core, then the height of each element is finite.


\begin{example}
    If the Collatz conjecture is true, the associated connected discrete dynamical system (Example \ref{ExampleCollatz}) is connected and has period $3$ and height $\infty$.
\end{example}
\begin{lemma}
    For every $(a,b)\in \NN$, the height of $\T{a}{b}$ is $a$ and the period is $b$.
\end{lemma}
\begin{proof}
    It is straightforward by checking each case.
\end{proof}


Actually, we have seen this generalized notion of height, in an implicit form in Proposition \ref{PropositionEventualBijConstruction}.
\begin{proposition}\label{PropositionEventualBijIntermsOfHeight}
    The quotient $Q$ constructed by Proposition \ref{PropositionEventualBijConstruction} and $a\in \N$ is exactly the class of discrete dynamical systems whose connected component's height is at most $a$.
\end{proposition}
\begin{proof}
    It is enough to prove that, for a connected discrete dynamical system $\X = (X,f)$, $\X$ belongs to the quotient $Q$ if and only if the height of $\X$ is at most $a$.

    We first prove \dq{if} part. If the height of $\X$ is at most $a$, $\X$ has a core $C \subset \X$ and $f^{a}(X) = C$. Since $f$ is a self-bijection on the core $C$, this proves that $\X$ belongs to $Q$.

    The \dq{only if} part follows from the fact that, if $\X$ belongs to $Q$, then $C=\Image{f^a}$ defines a core of $\X$.
\end{proof}

\subsection{Generative order of connected discrete dynamical systems}\label{SubsectionConnectedIsom}
In this subsection, we prove that the function $\NN \to \Gc_{\DD}$ is an isomorphism of posets (Proposition \ref{PropositionConnectedIsomorphic}). We divide our proof into three parts.
\begin{enumerate}
    \item order-preserving (part \ref{subsubsectionOrderpreserving})
    \item surjective (part \ref{subsubsectionSurjective})
    \item embedding (part \ref{subsubsectionEmbedding})
\end{enumerate}

Before going into these parts, we prepare two tools for 
constructions, so that we can relate period, height and the generative order. The first lemma is related to period, and the second is to height. For both lemmas, we utilize the canonical endomorphism (Example \ref{ExampleCanonicalEndo}).

\invmemo{Extracting, period, and height respectively. maybe, these can be in the previous subsection. Canonical endo}

\begin{lemma}[Period and generative order]
\label{LemmaEventualImage}
    For a connected discrete dynamical system $\X=(X,f)$ with the period $b\in \cpN$, we have $\X \goeq \T{0}{b}$.
\end{lemma}
\begin{proof}
    We prove the colimit of the chain of the canonical endomorphism (Example \ref{ExampleCanonicalEndo})
    \[
    \begin{tikzcd}
        \cdots \ar[r,"f"] & \X\ar[r,"f"] & \X\ar[r,"f"] & \X\ar[r,"f"] & \cdots 
    \end{tikzcd}
    \]
    is isomorphic to $\T{0}{b} \cong (\Z/b\Z,+1)$. The underlying set of the colimit is a set $(X \times \Z)/ \sim$, where the equivalence relation $\sim$ is generated by
    \[(x,n+1) \sim (f(x),n).\]
    The associated endofunction is defined by $(x,n)\mapsto (f(x),n)$, or equivalently, $(x,n)\mapsto (x,n+1)$.

    Take an arbitrary element $x\in X$ (using the non-emptiness, implied by the connectedness).
    By the connectedness condition \ref{ConditionConcreteConnectedness} in Definition \ref{DefinitionConnectedDD}, every element is equivalent to $(x,n)$ for some $n\in \Z$. Furthermore, it is straightforward to prove that $(x,n) \sim (x,m)$ holds if and only if $n-m \in b\Z$.
\end{proof}
\begin{remark}[As a Kan extension]\label{RemarkAsAKanExtension}
    This is a canonical way to make a discrete dynamical system into \dq{invertible}
    system, i.e., this is the left Kan extension of $\N \to \Set$ along $\N \to \Z$.
\end{remark}


\begin{lemma}[Height filtration]\label{LemmaHeightfiltration}
    Let $\X = (X,f)$ be a connected discrete dynamical system with a core and $\X_n$ denote the subobject that consists of elements whose height is less than or equal to $n$. Then, we have 
    \[
    \X = \sup_{n \in \N} \X_{n}
    \]
    in $\G_{\DD}$. 
\end{lemma}
\begin{proof}
    Notice that the core is none other than $\X_0$. By Lemma \ref{LemmaEventualImage}, we have $\X\goeq \T{0}{b}\cong \X_0$. 
    Since $\X_{n+1}$ is the inverse image of $\X_{n}$ along $f$, we obtain the following chain of pullbacks.
    \[
    \begin{tikzcd}
        \cdots\ar[r]&
        \X_3 \ar[r]\ar[d,rightarrowtail]\arrow[dr, phantom, "\lrcorner", very near start]&
        \X_2 \ar[r]\ar[d,rightarrowtail]\arrow[dr, phantom, "\lrcorner", very near start]&
        \X_1 \ar[r]\ar[d,rightarrowtail]\arrow[dr, phantom, "\lrcorner", very near start]&
        \X_{0}\ar[d,rightarrowtail]\\
        \cdots\ar[r,"f"]&
        \X\ar[r,"f"]&
        \X\ar[r,"f"]&
        \X\ar[r,"f"]&
        \X
    \end{tikzcd}
    \]
    By induction, we have $\X\goeq \X_n$. For the other direction $\X\ogeq \sup_{n \in \N} \X_{n}$, we can use the filtered colimit
    \[
        \begin{tikzcd}
            \X_0\ar[r]&\X_1\ar[r]&\X_2\ar[r]&\cdots &\hspace{-20pt}\X.
        \end{tikzcd}
    \]
    Here we used the fact that the height of an element of $\X$ is finite, as mentioned right after Definition \ref{DefinitiongeneralizedHeight}.
\end{proof}

\subsubsection{Order-preserving}\label{subsubsectionOrderpreserving}
In this part, we prove the correspondence $\NN \to \Gc_{\DD}$ is order-preserving. 
Our strategy is dividing our problem into \dq{height part} and \dq{period part,} verified by the next lemma.

\begin{lemma}\label{LemmaDecompose}
    For every $(a,b)\in \NN$, we have $\T{a}{b} \gog \T{a}{1} + \T{0}{b}$ in $\G_{\DD}$. 
\end{lemma}
\begin{proof}
By Proposition \ref{PropositionFullCoprod}, it is enough to prove the following three relations:
    \begin{itemize}
        \item $\T{a}{b} \goeq \T{0}{b}$
        \item $\T{a}{b} \goeq \T{a}{1}$
        \item $\T{a}{1} + \T{0}{b}\goeq \T{a}{b}$
    \end{itemize}
    \begin{description}
        \item[$\T{a}{b} \goeq \T{0}{b}$] This immediately follows from Lemma \ref{LemmaEventualImage}.
        \item[$\T{a}{b} \goeq \T{a}{1}$] Consider the following pushout diagram
        \[
        \begin{tikzcd}
            \T{0}{b}\ar[r]\ar[d]&\T{a}{b}\ar[d]\\
            \1\ar[r]&\T{a}{1},\arrow[lu, phantom, "\rotatebox{180}{$\lrcorner$}", very near start]
        \end{tikzcd}
        \]
        where $\1$ denotes the terminal object. Notice that even for the cases where $a=\infty$ or $b=0$, this construction does work.
        
        \item[$\T{a}{1} + \T{0}{b}\goeq \T{a}{b}$] 
        From the fact that the prequotient is closed under binary products and from Lemma \ref{LemmaCoproduct}, it follows that $\T{a}{1} + \T{0}{b}\goeq \T{a}{1}\times \T{0}{b}$.
        It is enough to prove $\T{a}{1}\times \T{0}{b}\goeq \T{a}{b}$. Actually, by a straightforward argument, we can prove $\T{a}{b}$ is a retract of $\T{a}{1}\times \T{0}{b}$. Intuitively, the retraction is constructed by \dq{leaving just one branch standing while winding up all the other branches} (see Figure \ref{FigureRetract}). Using Lemma \ref{LemmaRetract} finishes the proof.
        \begin{figure}[ht]
    \begin{shaded}
        \centering
        \begin{tikzpicture} [scale = \scvalue]
        \def \r{2/pi}
        
        \def \a{0}
        \node at (-5, \a -0.5) {$\T{3}{4}\colon$};
        \foreach \x in {-3, -2, -1} {
            \draw[black, thick, \arst] (\x, \a) -- (\x + 1, \a);
            \fill[black] (\x,\a) circle (0.06);
        }
        \draw[black, thick, \arst] (0,\a)     arc (90  :  0:\r);
        \draw[black, thick, \arst] (\r,\a-\r) arc (360 :270:\r);
        \draw[black, thick, \arst](0, \a-2*\r)arc (270 :180:\r);
        \draw[black, thick, \arst] (-\r,\a-\r)arc (180 : 90:\r);
        \fill[black] (0,\a) circle (0.06);
        \fill[black] (\r,\a-\r) circle (0.06);
        \fill[black] (0, \a-2*\r) circle (0.06);
        \fill[black] (-\r,\a-\r) circle (0.06);

        \def \b{-4.5}
        \node at (-5, \b -0.5) {$\T{3}{0}\times \T{0}{4}\colon$};
        \foreach \x in {1,2,3} {
            \draw[black, thick, \arst] (-\x, \b) -- (1- \x, \b);
            \fill[black] (-\x,\b) circle (0.06);

            \draw[black, thick, \arst] (\r, \b - \r + \x) -- (\r, \b - \r + \x -1);
            \fill[black] (\r, \b - \r + \x) circle (0.06);

            \draw[black, thick, \arst] (0 + \x, \b - 2 * \r) -- (0 + \x -1, \b - 2 * \r);
            \fill[black] (0 + \x, \b - 2 * \r) circle (0.06);

            \draw[black, thick, \arst] (0 + \x, \b - 2 * \r) -- (0 + \x -1, \b - 2 * \r);
            \fill[black] (0 + \x, \b - 2 * \r) circle (0.06);

            \draw[black, thick, \arst] (- \r, \b - \r - \x) -- (- \r, \b - \r - \x +1);
            \fill[black] (- \r, \b - \r - \x) circle (0.06);
        }
        \draw[black, thick, \arst] (0,\b)     arc (90  :  0:\r);
        \draw[black, thick, \arst] (\r,\b-\r) arc (360 :270:\r);
        \draw[black, thick, \arst](0, \b-2*\r)arc (270 :180:\r);
        \draw[black, thick, \arst] (-\r,\b-\r)arc (180 : 90:\r);
        \fill[black] (0,\b) circle (0.06);
        \fill[black] (\r,\b-\r) circle (0.06);
        \fill[black] (0, \b-2*\r) circle (0.06);
        \fill[black] (-\r,\b-\r) circle (0.06);

        \end{tikzpicture}
        
        \caption{Visualization of $\T{3}{4}$ and $\T{3}{0}\times \T{0}{4}$}
        \label{FigureRetract}
    \end{shaded}
\end{figure}
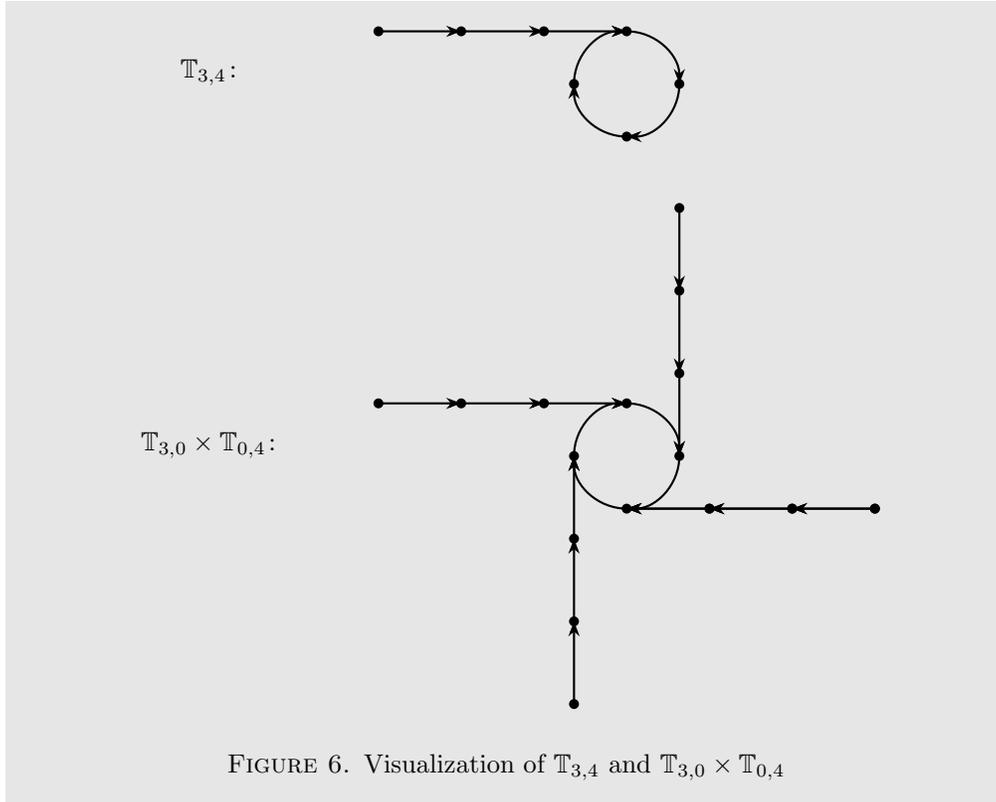
        
    \end{description}
\end{proof}

\begin{proposition}\label{PropositionOrderPreserivng}
    The correspondence $\NN \to \Gc_{\DD}$ is order-preserving.
\end{proposition}
\begin{proof}
    By Lemma \ref{LemmaDecompose}, it is enough to prove the following two statements:
    \begin{enumerate}
        \item If $a_0 \leq a_1$ in $\cN$, then $\T{a_0}{1} \ogeq \T{a_1}{1}$.
        \item If $b_0 \leq b_1$ in $\cpN$, then $\T{0}{b_0} \ogeq \T{0}{b_1}$.
    \end{enumerate}
    We prove them one by one.
        \begin{enumerate}
        \item This follows from the height filtration (Lemma \ref{LemmaHeightfiltration}). 
        

        \item Suppose $b_0 \leq b_1$ in $\cpN$, which means $b_0$ divides $b_1$.
        Then, $\T{0}{b_0}$ is constructed as a coequalizer
        \[
        \begin{tikzcd}
            \T{0}{b_1} \ar[r,shift left, "f^{b_0}"] \ar[r,shift right, "\id"'] & \T{0}{b_1} \ar[r, twoheadrightarrow] & \T{0}{b_0},
        \end{tikzcd}
        \]
        where $f$ denotes the associated endomorphism of $\T{0}{b_1}$. This construction does work even if $b_0=0$ (and thus $f^{b_0}=f^{0} = \id$).
    \end{enumerate}
\end{proof}

\subsubsection{Surjective}\label{subsubsectionSurjective}
In this part, we prove that the correspondence $\NN \to \Gc_{\DD}$ is surjective. This is the hardest part of the classification of quotients. It is essentially this part that ensures there are no omissions in our list of quotients.


The next lemma is technically convenient in that it allows us to reduce the construction of $\T{\infty}{b}$ to those of $\T{a}{b}$ for $a<\infty$.
\begin{lemma}\label{lemmaAinftySup}
    For any $b\in \cpN$, we have
    \[\T{\infty}{b} = \sup_{a<\infty} \T{a}{b}\]
    in $\G_{\DD}$.
    
\end{lemma}
\begin{proof}
If $b\neq 0$, this is just an example of the height filtration (Lemma \ref{LemmaHeightfiltration}). By Lemma \ref{LemmaDecompose}, we also have 
\[\T{\infty}{0}= \T{\infty}{1} + \T{0}{0} = \sup_{a<\infty} (\T{a}{1} + \T{0}{0}) = \sup_{a<\infty} \T{a}{0}.\]
\end{proof}


\begin{proposition}\label{PropositionSurjection}
    For a connected discrete dynamical system $\X=(X,f)$ with height $a$ and period $b$, we have 
    \[\X \gog \T{a}{b}.\]
    In particular, the correspondence $\NN \to \Gc_{\DD}$ is surjective.
\end{proposition}
\begin{proof}
We divide our proof into several parts.
    \begin{description}
        \item[If $a$ is finite]\hspace{1pt}\\
        Suppose the height $a$ is finite. In this case, considering the element with the maximum height $a$, we can prove that $\T{a}{b}$ is a retract of $\X$. Therefore, $\X \goeq \T{a}{b}$ by Lemma \ref{LemmaRetract}. 
        
        In the other direction, we prove $\T{a}{b}\goeq \X$. Actually, $\X$ is a colimit of $\T{a}{b}$.
        \[
        \X \cong \colim_{x\in X} \T{a}{b}
        \]
        Although a direct proof is not very hard, there is a much easier way to prove this, using the monoid theoretic background (subsection \ref{SubsectionEssentialQuotients}). 

        \item[If $\X$ has a core]\hspace{1pt}\\
        This case is reduced to the above case by the height filtration (Lemma \ref{LemmaHeightfiltration}) and Lemma \ref{lemmaAinftySup}. In detail, by adopting the notation in the lemma, we obtain
        \[
        \X = \sup_{n\in \N} \X_{n} = \sup_{n \in \N} \T{\min(n,a)}{b} = \T{a}{b}
        \]
        in the poset $\G_{\DD}$.

        \item[If $\X$ does not have a core]\hspace{1pt}\\
         Since $\T{\infty}{0} \gog \Nsucc$ is the maximum element (Lemma \ref{LemmaNsuccisMaximum}), $\X \ogeq \T{\infty}{0}$ is easy. We prove the other direction $\X \goeq \T{\infty}{0}$. 
        
        If there exist two or more subobjects that are isomorphic to $\T{0}{0}  \cong (\Z, +1)$,
        then there exist two morphisms $\T{0}{0} \rightrightarrows \X$ such that their equalizer is $\Nsucc \subset (\Z,+1) \cong \T{0}{0}$. 
        Combining with $\X \goeq \T{0}{0}$ (Lemma \ref{LemmaEventualImage}), we obtain $\X\goeq \Nsucc \simeq \T{\infty}{0}$. 

        What remains is the case where there is no morphism from $\T{0}{0}\cong (\Z, +1)$ to $\X$. This is the non-trivial part! Until the end of the proof, we assume $\X$ satisfies that condition. In other words, we assume
        that there are no \dq{going back} infinite sequences
        \[x\mapsfrom x'\mapsfrom x''\mapsfrom x'''\mapsfrom \cdots.\]
        
        Under this \dq{well-foundedness} assumption, we can recursively define a function $\w\colon X \to \cN$ by 
        \[
        \w(x) =\sup\{\w(y) +1\mid y\in f^{-1}(x)\} .
        \]
        For example, if $f^{-1}(x) = \emptyset$, then $\w(x) = \sup \emptyset = 0$. 
        Intuitively, $\w(x)$ represents the height of the subtree above $x$
        (see Figure \ref{PictureWeight}). 
        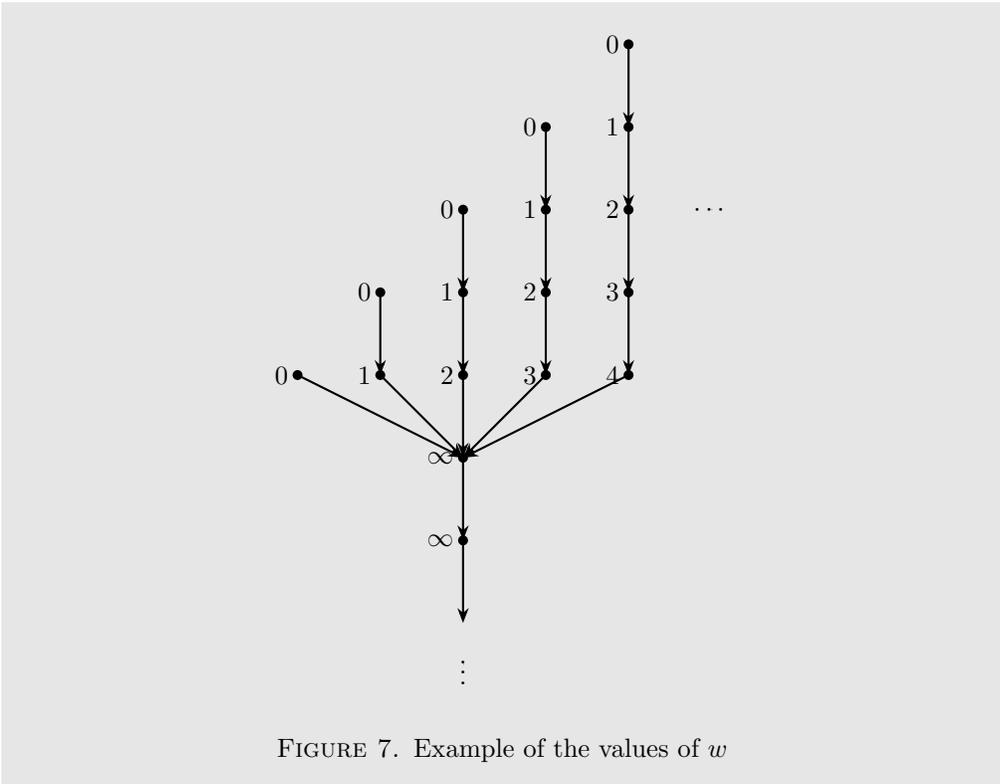
\begin{figure}[ht]
    \begin{shaded}
    \centering
        \begin{tikzpicture} [scale = \scvalue]
            \fill[black] (0,0) circle (0.06) node[left]{$0$};
            \draw[black, thick,\arst](0,0)--(2,-1);
            
            \fill[black] (1,0) circle (0.06) node[left]{$1$};
            \draw[black, thick,\arst](1,0)--(2,-1);
            \fill[black] (1,1) circle (0.06) node[left]{$0$};
            \draw[black, thick,\arst](1,1)--(1,0);
            
            \fill[black] (2,0) circle (0.06) node[left]{$2$};
            \draw[black, thick,\arst](2,0)--(2,-1);
            \fill[black] (2,1) circle (0.06) node[left]{$1$};
            \draw[black, thick,\arst](2,1)--(2,0);
            \fill[black] (2,2) circle (0.06) node[left]{$0$};
            \draw[black, thick,\arst](2,2)--(2,1);
            
            \fill[black] (3,0) circle (0.06) node[left]{$3$};
            \draw[black, thick,\arst](3,0)--(2,-1);
            \fill[black] (3,1) circle (0.06) node[left]{$2$};
            \draw[black, thick,\arst](3,1)--(3,0);
            \fill[black] (3,2) circle (0.06) node[left]{$1$};
            \draw[black, thick,\arst](3,2)--(3,1);
            \fill[black] (3,3) circle (0.06) node[left]{$0$};
            \draw[black, thick,\arst](3,3)--(3,2);
            
            \fill[black] (4,0) circle (0.06) node[left]{$4$};
            \draw[black, thick,\arst](4,0)--(2,-1);
            \fill[black] (4,1) circle (0.06) node[left]{$3$};
            \draw[black, thick,\arst](4,1)--(4,0);
            \fill[black] (4,2) circle (0.06) node[left]{$2$};
            \draw[black, thick,\arst](4,2)--(4,1);
            \fill[black] (4,3) circle (0.06) node[left]{$1$};
            \draw[black, thick,\arst](4,3)--(4,2);
            \fill[black] (4,4) circle (0.06) node[left]{$0$};
            \draw[black, thick,\arst](4,4)--(4,3);

            \fill[black] (2,-1) circle (0.06) node[left]{$\infty$};
            \draw[black, thick,\arst](2,-1)--(2,-2);
            \fill[black] (2,-2) circle (0.06) node[left]{$\infty$};
            \draw[black, thick,\arst](2,-2)--(2,-3);
            \fill[black] (2,-3.5) circle (0) node[]{$\vdots$};

            \fill[black] (5,2) circle (0) node[]{$\cdots$};
        \end{tikzpicture} 
    \caption{Example of the values of $\w$}
    \label{PictureWeight}
    \end{shaded}
\end{figure}
        
        A first observation on $\w$ is that 
        $\Image{\w}\cap \N$ is unbounded in $\N$.
        If $\infty \in \Image{\w}$, then using well-foundedness we can find an element $x\in \w^{-1}(\infty)$ such that $\{\w(y)\mid y\in f^{-1}(x)\}\subset \N$. 
        Then, $\{\w(y)\mid y\in f^{-1}(x)\}\subset \N$ is unbounded.
        If $\infty \notin \Image{\w}$, taking an arbitrary element $x\in X$, we obtain a strictly increasing sequence of natural numbers
        \[\w(x)< \w(f(x))< \w(f^2 (x))< \cdots ,\]
        which witnesses that $\Image{\w}\cap \N$ is unbounded.

        Next, we prove $\X \goeq \T{n}{0}$ for an arbitrary natural number $n\in \N$. 
        Since $\Image{\w}\cap \N$ is unbounded in $\N$, we can take $x\in X$ such that $n\leq \w(x) < \infty$.
        If $\w(f^{k+1} (x))= \w(f^{k}(x))+1$ for every $k\in \N$ (in particular $\Image{\w}\subset \N$), there is a split monomorphism $i\colon \Nsucc \to \X$ such that $i(\w(x))=x$. 
        \invmemo{write more}
        Then, in such a case, $\X \goeq \Nsucc \goeq \T{n}{0}$. If it is not the case (i.e., $\w(f^{k+1} (x))> \w(f^{k}(x))+1$ for some $k\in \N$), replacing $x$ with $f^{k}(x)$, we can assume $\w(f(x))>\w(x)+1$. 

        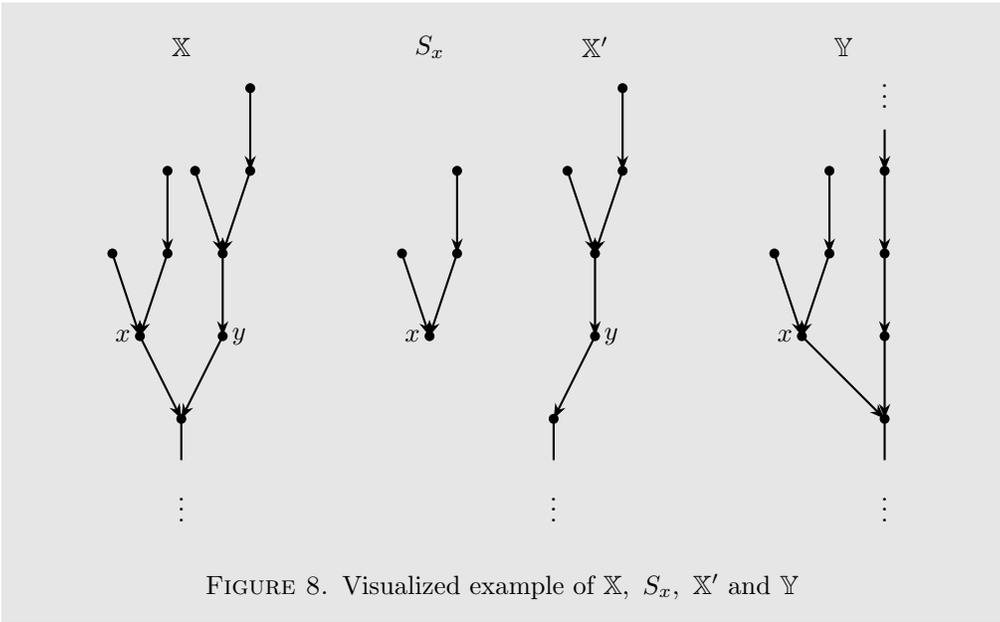
\begin{figure}[ht]
    \begin{shaded}
        \centering
        \begin{tikzpicture} [scale = \scvalue]
        \def \a{0}
        \def \b{0}
        \def \c{1}
        \def \d{1}
        \def \h{3.5}
        \def \s{0.333}
        
        \node at (\a+0.5, \h) {$\X$};
        \fill[black] (\a,\b) circle (0.06) node[left] {$x$};
        \fill[black] (\a-\s,\b+1) circle (0.06);
        \draw[black, thick, \arst] (\a-\s,\b+1) -- (\a,\b);
        \fill[black] (\a+\s,\b+1) circle (0.06);
        \draw[black, thick, \arst] (\a+\s,\b+1) -- (\a,\b);
        \fill[black] (\a+\s,\b+2) circle (0.06);
        \draw[black, thick, \arst] (\a+\s,\b+2) -- (\a+\s,\b+1);
        \fill[black] (\a+\c,\b+\d) circle (0.06);
        \draw[black, thick, \arst] (\a+\c,\b+\d) -- (\a+\c,\b);
        \fill[black] (\a+\c,\b) circle (0.06) node[right] {$y$};
        \fill[black] (\a-\s+\c,\b+\d+1) circle (0.06);
        \draw[black, thick, \arst] (\a-\s+\c,\b+\d+1) -- (\a+\c,\b+\d);
        \fill[black] (\a+\s+\c,\b+\d+1) circle (0.06);
        \draw[black, thick, \arst] (\a+\s+\c,\b+\d+1) -- (\a+\c,\b+\d);
        \fill[black] (\a+\s+\c,\b+\d+2) circle (0.06);
        \draw[black, thick, \arst] (\a+\s+\c,\b+\d+2) -- (\a+\s+\c,\b+\d+1);
        \fill[black] (\a+0.5,\b-1) circle (0.06);
        \draw[black, thick] (\a+0.5,\b-1) -- (\a+0.5,\b-1.5);
        \node at (\a+0.5,\b-2) {$\vdots$};
        \draw[black, thick, \arst] (\a+\c,\b) -- (\a+0.5,\b-1);
        \draw[black, thick, \arst] (\a,\b) -- (\a+0.5,\b-1);

        \def \sa{3.5}
        \node at (\sa, \h) {$S_x$};
        \fill[black] (\sa,\b) circle (0.06) node[left] {$x$};
        \fill[black] (\sa-\s,\b+1) circle (0.06);
        \draw[black, thick, \arst] (\sa-\s,\b+1) -- (\sa,\b);
        \fill[black] (\sa+\s,\b+1) circle (0.06);
        \draw[black, thick, \arst] (\sa+\s,\b+1) -- (\sa,\b);
        \fill[black] (\sa+\s,\b+2) circle (0.06);
        \draw[black, thick, \arst] (\sa+\s,\b+2) -- (\sa+\s,\b+1);

        \def \da{4.5}
        \node at (\da+1, \h) {$\X'$};
        \fill[black] (\da+\c,\b+\d) circle (0.06);
        \draw[black, thick, \arst] (\da+\c,\b+\d) -- (\da+\c,\b);
        \fill[black] (\da+\c,\b) circle (0.06) node[right] {$y$};
        \fill[black] (\da-\s+\c,\b+\d+1) circle (0.06);
        \draw[black, thick, \arst] (\da-\s+\c,\b+\d+1) -- (\da+\c,\b+\d);
        \fill[black] (\da+\s+\c,\b+\d+1) circle (0.06);
        \draw[black, thick, \arst] (\da+\s+\c,\b+\d+1) -- (\da+\c,\b+\d);
        \fill[black] (\da+\s+\c,\b+\d+2) circle (0.06);
        \draw[black, thick, \arst] (\da+\s+\c,\b+\d+2) -- (\da+\s+\c,\b+\d+1);
        \fill[black] (\da+0.5,\b-1) circle (0.06);
        \draw[black, thick] (\da+0.5,\b-1) -- (\da+0.5,\b-1.5);
        \node at (\da+0.5,\b-2) {$\vdots$};
        \draw[black, thick, \arst] (\da+\c,\b) -- (\da+0.5,\b-1);

        \def \ya{8}
        \node at (\ya +0.5, \h) {$\Y$};
        \fill[black] (\ya,\b) circle (0.06) node[left] {$x$};
        \fill[black] (\ya-\s,\b+1) circle (0.06);
        \draw[black, thick, \arst] (\ya-\s,\b+1) -- (\ya,\b);
        \fill[black] (\ya+\s,\b+1) circle (0.06);
        \draw[black, thick, \arst] (\ya+\s,\b+1) -- (\ya,\b);
        \fill[black] (\ya+\s,\b+2) circle (0.06);
        \draw[black, thick, \arst] (\ya+\s,\b+2) -- (\ya+\s,\b+1);

        \foreach \i in {-1, 0, 1}{
            \draw[black, thick, \arst] (\ya+1,\i+1) -- (\ya+1,\i);
        }
        \draw[black, thick] (\ya+1,-1) -- (\ya+1,-1.5);
        \node at (\ya+1,-2) {$\vdots$};
        \draw[black, thick, \arst] (\ya+1,2.5) -- (\ya+1,2);
        \node at (\ya+1,3) {$\vdots$};
        \foreach \i in {-1, 0, 1, 2}{
            \fill[black] (\ya+1,\i) circle (0.06);
        }
        \draw[black, thick, \arst] (\ya,0) -- (\ya+1,-1);

        \end{tikzpicture}
        
        \caption{Visualized example of $\X,\ S_{x},\ \X'$ and $\Y$}
        \label{FigureXY}
    \end{shaded}
\end{figure}
        For the following construction, see Figure \ref{FigureXY}. Let $S_x$ denote the subset of $X$ defined by
        \[S_x =\{x'\in X\mid \exists n\in \N,\ f^{n}(x')=x\},\]
        and $\X'$ denotes the subobject of $\X$ whose underlying set is $X\setminus S_x$.
        In other words, $S_x$ is the subset of all nodes above $x$, and $\X'$ is the maximum subobject that does not contain $x$.
        Then, we can prove $\X \goeq \X'$. In fact, by definition of $x$, $w(f(x))>w(x)+1$, so there is another element $y\neq x$ such that $f(x)=f(y)$ and $\w(x)< \w(y)$. Then, since $w(x)$ and $w(y)$ are the heights of subtree above $x$ and $y$, there is an endomorphism $g\colon \X \to \X$ such that $g(x)=y$ and the following diagram 
        \[
        \begin{tikzcd}
            \X'\ar[r,rightarrowtail]&\X\ar[r,"g",shift left]\ar[r,"\id_{\X}"',shift right]&\X
        \end{tikzcd}
        \]
        is an equalizer diagram, which proves $\X \goeq \X'$. 
        (The endomorphism $g$ can be constructed as a function that keeps $\X'$ as it is and maps $S_x$ to the longest path over $y$.)
        
        Then, we can construct the following pushout diagram
         \[
        \begin{tikzcd}
            \X'\ar[r]\ar[d]&\T{0}{0}\ar[d]\\
            \X\ar[r]&\Y,\arrow[lu, phantom, "\rotatebox{180}{$\lrcorner$}", very near start]
        \end{tikzcd}
        \]
        where $\X'\to \X$ is the canonical inclusion and $\X'\to \T{0}{0}$ is the essentially unique morphism. Since $\X\goeq \X', \T{0}{0}$, we have $\X\goeq \Y$. Furthermore, $\T{\w(x)}{0}$ is a retract of $\Y$. This follows from the concrete calculation of the pushout:
        the underlying set of $\Y$ is given by $\Z \coprod S_x$, and the associated endofunction sends $x\in S_x$ to $0\in \Z$.
        Then, by Lemma \ref{LemmaRetract} and Proposition \ref{PropositionOrderPreserivng}, we have 
        \[\X \goeq \Y \goeq \T{\w(x)}{0}\goeq \T{n}{0}.\]
        Since we took $n$ arbitrarily,  by Lemma \ref{lemmaAinftySup}, we have $\X\goeq \T{\infty}{0}$.
    \end{description}
\end{proof}
\invmemo{On recursive coalgebra}

\subsubsection{Embedding}\label{subsubsectionEmbedding}
At last, we prove that the correspondence is an embedding and hence, an isomorphism. In this part, we use two constructions of quotients: Proposition \ref{PropositionConstructionPrimeQuotients} and Proposition \ref{PropositionEventualBijConstruction}.
\begin{proposition}\label{PropositionConnectedIsomorphic}
    The correspondence $ \NN \to \Gc_{\DD}$ is an isomorphism of posets.
\end{proposition}
\begin{proof}
    So far, we have proven that it is an order-preserving surjection. 
    Now, what we should prove is that it is order-reflecting.

    Suppose $\T{a}{b} \goeq \T{a'}{b'}$. We will prove $a\geq a'$ in $\cN$ and $b\geq b'$ in $\cpN$. 
    \begin{description}
        \item[$a\geq a'$ in $\cN$] 
        If $a=\infty$, then the statement is trivial. We assume $a<\infty$ and let $Q$ be the quotient constructed in Proposition \ref{PropositionEventualBijConstruction}. Then, $Q$ contains $\T{a}{b}$ and hence $\T{a'}{b'}$ as well. This implies $a\geq a'$.
        
        \item[$b\geq b'$ in $\cpN$]
        If $b=0$, then the statement is trivial. We assume $b\neq 0$ and let $Q$ be the quotient constructed by Proposition \ref{PropositionConstructionPrimeQuotients} and $\alpha \colon \cP \to \cN$ that sends $\infty \in \cP$ to $\infty \in \cN$ and a (finite) prime number $p\in \cP$ to the number of times $p$ divides $b$. Then, $Q$ contains $\T{a}{b}$ and hence $\T{a'}{b'}$ as well. This implies $b\geq b'$ in $\cpN$.
        

    \end{description}
\end{proof}

\subsection{Main theorem}\label{subsectionMainTheorem}
So far, we have completely determined the structure of $\Gc_{\DD} $. By Proposition \ref{PropositionChainOfEmbeddings}, all that remains is to determine when the downward closed set of $\Gc_{\DD} $ defines a quotient of $\DD$.

\begin{lemma}\label{LemmaSuppreserving}
    The embedding function
    \[\NN \cong \Gc_{\DD} \to \G_{\DD} \cong \PQ_{\DD}\]
    preserves finite supremums.
\end{lemma}
\begin{proof}
    By Lemma \ref{LemmaDecompose} and Proposition \ref{PropositionOrderPreserivng}, the proof is already almost done, and the remaining relations we have to prove are
    \begin{itemize}
        \item $\T{0}{1}$ is minimum, and
        \item $\T{0}{b}+ \T{0}{b'}\goeq \T{0}{\Lcm(b, b')}$.
    \end{itemize}
    The first relation follows from Example \ref{ExampleMinimumElement}. We prove the second one.
    If $b$ or $b'$ is $0$, $\Lcm(b,b')$ is also $0$, and this is trivial. If $b,b'\neq 0$, since $\T{0}{b}\times \T{0}{b'}$ is $\Gcd(b,b')$-copies of $\T{0}{\Lcm(b,b')}$, we have $\T{0}{b}+ \T{0}{b'}\goeq \T{0}{\Lcm(b,b')}$.
\end{proof}

Before stating our main theorem, we clarify the usage of the term \dq{ideals} of a poset.
\begin{definition}[Ideal]\label{DefinitionIdeal}
    For a poset $P$ with finite supremums, an \emph{ideal} of $P$ is a downward closed subset $I \subset P$ that is closed under finite supremums.
\end{definition}
Notice that an ideal is non-empty since it contains the supremum of the empty set, i.e., the minimum element.
\begin{notation}
    The poset of all ideals of a poset $P$ with the inclusion relation is denoted by $\Ideal (P)$.
\end{notation}

\begin{theorem}[Main Theorem]\label{TheoremMainTheorem}
    There is a natural bijective correspondence between 
    \begin{itemize}
        \item Quotient toposes of $\DD$ and
        \item Ideals of $\rNN$.
    \end{itemize}
    Furthermore, it defines the poset isomorphism
    \[
    \PQ_{\DD}\cong \Q_{\DD} \cong \Ideal (\rNN)
    \]
\end{theorem}
\begin{proof}
    Take an arbitrary prequotient $Q$ of $\DD$. Then by Proposition \ref{PropositionChainOfEmbeddings}, we obtain the associated downward closed set $J\subset \NN \cong \Gc_{\DD}$. By Lemma \ref{LemmaSuppreserving}, $J$ defines an ideal of $\Gc_{\DD}$.
    Then, $I=J\cap (\rNN)\subset \rNN$ is an ideal of $\rNN$. 
    By Lemma \ref{lemmaAinftySup}, $J$ contains $(\infty, b) \in \NN$ if and only if $I$ contains $(a,b) \in \rNN$ for all $a \in \N$. This proves that the composite correspondence
    \[
    \begin{tikzcd}
        \PQ_{\E} \ar[r]& \Ideal (\NN) \ar[r
        ]&\Ideal (\rNN)
    \end{tikzcd}
    \]
    is injective.

    Conversely, take an arbitrary ideal $I\subset \rNN$. It suffices to prove $I$ is induced by a quotient $Q$. In other words, we construct a quotient $Q$ such that 
    \[
    I = \{(a,b)\in \rNN \mid 
    \T{a}{b} \in Q
    \}
    \]
    
    \begin{description}
        \item[If $(0,0)\in I$] \hspace{1pt}\\
        As $I$ is downward closed, we can take $A\in \cN$ such that
        \[(a,1)\in I \iff 
        a\leq A.\]
        Then, since $I$ is an ideal, we have \[I=\{a\in \N\mid a\leq A\}\times \cpN.\]
        If $A = \infty$, $I$ is induced by the maximum quotient, which is $\DD$ itself. If $A< \infty$, $Q$ is constructed by Proposition \ref{PropositionEventualBijConstruction} and a non-negative integer $A\in \N$.

        \item[If $(0,0)\notin I$] \hspace{1pt}\\
        In this case, for any $(a,b)\in I$, $b$ is not equal to $0$. We define $\alpha \colon \cP \to \cN$ by
        \begin{align*}
            \alpha(\infty) &= \sup\{n\in \N \mid (n,1)\in I\}\\
            \alpha(p) &= \sup\{n\in \N\mid (0,p^n)\in I\}.
        \end{align*}
        By Proposition \ref{PropositionConstructionPrimeQuotients}, there is the associated hyperconnected quotient $Q$ and the associated ideal of $Q$ coincides with $I$.
        \invmemo{We can write a bit more}
    \end{description}
\end{proof}

We finish this section with several immediate corollaries. As promised, by the proof of our main theorem, we also obtained Corollary \ref{CorollaryElementaryInAdvance}.

\begin{corollary}[Prequotient $=$ Quotient for $\DD$]\label{CorollaryDDprequotientIsQuotient}
    The class of discrete dynamical systems that is closed under finite limits and small colimits is precisely the same as a quotient of $\DD$. In other words, for the topos of discrete dynamical systems $\DD$, the notions of quotients and prequotients coincide. 
\end{corollary}

\begin{corollary*}[\ref{CorollaryElementaryInAdvance}]
    Every class of discrete dynamical systems that is closed under taking finite limits and small colimits defines a quotient of $\DD$ and obtained by one of the following constructions:
    \begin{itemize}
        \item $\DD$ itself
        \item Quotient via Prime numbers (Proposition \ref{PropositionConstructionPrimeQuotients})
        \item Quotient of Eventual bijections (Proposition \ref{PropositionEventualBijConstruction}).
    \end{itemize}
\end{corollary*}

\begin{corollary}\label{CorollaryNumberofQuotients}
    The number of quotients of $\DD$ is the continuum cardinality.
\end{corollary}



\section{Where do they come from?}\label{SectionWhere}
So far, we have completed the classification of all quotients of the topos $\DD$.
In this section, we will discuss where those quotients come from.

We use some notions related to geometric morphisms, including connected, hyperconnected, and essential geometric morphisms. For the definitions, see \ref{AppendixGeometricMorphisms}, and for details see \cite{johnstone2002sketchesv1, johnstone2002sketchesv2}.
\invmemo{Write a summary}

\subsection{Hyperconnected quotients and LSC}\label{SubsectionHyperconnectedQuotients}
\invmemo{hyperconnected cite}
In this subsection, we explain how we can deduce Proposition \ref{PropositionConstructionPrimeQuotients} from the first author's previous paper \cite{hora2023internal} on \emph{hyperconnected quotients}.
Here, a hyperconnected quotient of a topos $\E$ means (the equivalence class of) a hyperconnected geometric morphism from $\E$. In other words, it is a quotient that is closed under taking subobjects.

In the first author's previous paper \cite{hora2023internal}, the classification of hyperconnected quotients is given in terms of internal semilattices. As a simple example of its main theorem, we have the following proposition:
\begin{proposition}[\cite{hora2023internal} for a presheaf topos]\label{PropositionLSC}
For a small category $\C$, the presheaf of co-subobjects of representable presheaves
\[\Xi \colon \C \to \Set \colon c \mapsto \mathrm{coSub}(y(c))=\{y(c)\twoheadrightarrow q \}\]
admits an internal semilattice structure, given by the usual poset structure of $\mathrm{coSub}(y(c))$.
There is a natural bijection between 
\begin{itemize}
    \item the internal filters of $\Xi$ and
    \item hyperconnected quotients of the presheaf topos $\ps{\C}$.
\end{itemize}
\end{proposition}
For details of the definition of $\Xi$, internal filters, and the correspondence, see \cite{hora2023internal}. As a corollary, we obtain the following classification of hyperconnected quotients. The classification of hyperconnected quotients of $\DD$ is not new. More general cases are studied in \cite{rogers2023toposes} and \cite{rosenthal1982quotient}.

\begin{corollary}[Classification of Hyperconnected quotients]\label{CorollaryHyperconnectedQuotientsOfDD}
A hyperconnected quotient of the topos $\DD$ is either $\DD$ itself or one constructed by Proposition \ref{PropositionConstructionPrimeQuotients}.
\end{corollary}
\begin{proof}
In the case where $\C = \N$, $\Xi$ itself is a discrete dynamical system. Its underlying set is the set of all quotient objects of the representable presheaf $\Nsucc$. Therefore, we have 
\[
\Xi = \{\T{a}{b}\mid (a,b) \in \rNN, b>0\} \cup \{\Nsucc\}.
\]
This $\Xi$ is isomorphic to the opposite poset of $(\N \times \cpN _{>}) \cup \{\infty\}$, where $\cpN _{>}$ is the poset of all positive integers with the divisibility relation. The associated endofunction is given by 
\[
(a,b) \mapsto
\begin{cases}
    (a-1,b) & (a>0)\\
    (0,b) & (a=0).
\end{cases}
\]
and $\infty \mapsto \infty$. By Proposition \ref{PropositionLSC}, there is a bijection between hyperconnected quotients of $\DD$ and internal filters of $\Xi$. Using the opposite isomorphism with $(\N \times \cpN _{>}) \cup \{\infty\}$, they correspond to ideals of $(\N \times \cpN _{>} )\cup \{\infty\}$ that are \dq{internal}, i.e., closed under the action of the associated endofunction. 
One can easily observe that this \dq{being internal} condition is trivial in this case.

If an ideal $J$ contains $\infty \in (\N \times \cpN _{>} )\cup \{\infty\}$, $J=(\N \times \cpN _{>} )\cup \{\infty\}$ corresponds to $\DD$ itself. If $\infty \notin J$, $J \subset \N \times \cpN _{>}$ corresponds to the one constructed by Proposition \ref{PropositionConstructionPrimeQuotients} since the poset $\N \times \cpN _{>}$ is isomorphic to the poset of finite-support functions $\cP \to \N$ with the pointwise order, and its ideal corresponds to the \dq{pointwise sup} function $\cP \to \cN$.
\end{proof}



\subsection{Essential quotients, Lax-epi functors, and Monoid epimorphisms}\label{SubsectionEssentialQuotients}
In this subsection, we will give a classification of all \emph{essential quotients}, which includes Proposition \ref{PropositionEventualBijConstruction}. 
Using these arguments, we will answer the questions: \dq{Where did the $\rNN$ in $\Q_{\E}\cong \Ideal (\rNN)$ come from?}, \dq{Where did the generalized notion of height come from?}, and \dq{Where did $\T{a}{0}$ come from?}. The answer in this subsection is: \dq{They come from monoid epimorphisms from $\N$!}. 

Our arguments in this subsection are based on \cite{el2002simultaneously}, which gives a complete description of all essential quotients. The concrete calculation leads us to the study of lax-epi functors and monoid epimorphisms from $\N$.


\subsubsection{Essential quotients of a presheaf topos}
Let us start by clarifying our terminology.
\begin{definition}
    A quotient is called \textit{essential}, if the associated connected geometric morphism is essential, as a geometric morphism.
\end{definition}
See Appendix \ref{AppendixGeometricMorphisms} for related definitions.

A typical example of essential geometric morphisms is one induced by a functor. A functor between small categories $F:\C \to \D$ induces the precomposition functor
\[\ps{\C} \leftarrow \ps{\D}.\]
This induced functor admits both left and right adjoints, thus defining an essential geometric morphism. 

Essential quotients of a presheaf topos are completely determined as an immediate corollary of 
\cite{el2002simultaneously}.
\begin{theorem}[\cite{el2002simultaneously}]\label{TheoremEssentialOfPresheaf}
    For a small category $\C$, an arbitrary essential quotient of $\ps{\C}$ is induced by a functor from $\C$. In particular, every essential quotient is a presheaf topos.
\end{theorem}

Therefore, for the purpose of this section, it is enough to observe
when a functor $F\colon \N \to \D$ induces an essential and connected geometric morphism, i.e., when the precomposite functor 
\[\ps{\N} \leftarrow \ps{\D}\]
is fully faithful. The answer for general functor $F\colon \C \to \D$ is given in \cite{adamek2001functors}, which focuses on \emph{lax-epi functors}.

\invmemo{Sousa}

\invmemo{Cite J. R. Isbell, Epimorphisms and Dominions}


\invmemo{Write on solid rings}

\begin{definition}[Lax epimorphism]
    A $1$-cell $F:C\to D$ in a $2$-category $\mathcal{A}$ is a \emph{lax epimorphism}, if, for every object $E$, the hom functor
    \[\mathcal{A} (C,E)\leftarrow \mathcal{A} (D,E)\]
    is fully faithful. 
\end{definition}

In particular, a functor between small categories $F:\C \to \D$ is a lax epimorphism in the $2$-category of small categories $\Cat$, if, for any small category $\E$, the induced precomposition functor 
\[
\E^{\C} \leftarrow \E^{\D}
\]
is fully faithful. A lax-epi functor is also called by other names, a \emph{co-fully faithful functor}, an \emph{absolutely dense functor}, and a \emph{connected functor} in \cite{adamek2001functors, el2002simultaneously, nunes2022lax}.
\cite{adamek2001functors} shows that this condition is equivalent to what we need:
\begin{proposition}[\cite{adamek2001functors}]
    For a functor between small categories $F\colon \C \to \D$, the following conditions are equivalent.
    \begin{itemize}
        \item $F$ is a lax epimorphism in the $2$-category of small categories.
        \item The pre-composite functor
        \[
        \ps{\C} \leftarrow \ps{\D}
        \]
        is fully faithful.
    \end{itemize}
\end{proposition}

\invmemo{citation rules for "Ada+"}

\subsubsection{Essential quotients of a (commutative) monoid action topos}
For the purposes of this paper, we are interested in lax epimorphisms from the monoid $\N$.

\begin{lemma}(\cite{adamek2001functors})\label{LemmaEssentiallySurjUpToRetract}
    A lax-epi functor $F\colon \C \to \D$ is essentially surjective up to retracts, i.e., every object of $\D$ is a retract of $Fc$ for some object $c$ of $\C$.
\end{lemma}

\begin{proposition}
    For a monoid $N$, every essential quotient of $\ps{N}$ is induced by a monoid homomorphism $\ph\colon N \to M$.
\end{proposition}
\begin{proof}
    Due to Theorem \ref{TheoremEssentialOfPresheaf}, every essential quotient of $\ps{M}$ is induced by a lax-epi functor $F\colon N \to \D$. Let $\bullet$ denote the unique object of $N$. By Lemma \ref{LemmaEssentiallySurjUpToRetract}, the embedding of the endomorphism monoid $M\coloneqq \D(F\bullet, F\bullet)$ to $\D$ is a Cauchy equivalence (i.e., it induces the equivalence between the presheaf toposes). This proves that the monoid homomorphism $\ph\coloneqq F\colon N \to M (\subset \D)$ induces the given essential quotient.
\end{proof}

Thus, our goal is to determine monoid homomorphisms $\ph\colon \N \to M$ that are lax epimorphisms.
In \cite{adamek2001functors}, an elementary condition for a functor to be a lax epimorphism is given. As a special case of the result, we obtain an elementary characterization of a lax epimorphism between monoids, using the next \dq{tensor} notation.
\begin{definition}\label{DefinitionTensor}
    For a monoid homomorphism $\ph \colon N \to M$, let $M\otimes_{N}M$ denote the quotient set of $M\times M$ by the equivalent relation $\sim$ generated by
    \[(m_0 , \ph(n)\ml  m_1)\sim (m_0\ml \ph(n), m_1)\]
    for each $n\in N$ and $m_0,m_1\in M$. The equivalence class of $(m_0, m_1)$ is denoted $m_0 \otimes_N m_1$.
\end{definition}
\invmemo{Difference from the coproduct}

Although the assumption of commutativity is not necessary for some implications, we assume it in the next proposition just for simplicity.
\begin{proposition}[Specialized case of \cite{adamek2001functors, isbell1966epimorphisms}]
\label{PropositionEquivalenceOfEpiLaxEpi}
For a commutative monoid $N$ and a monoid homomorphism $\ph\colon N\to M$, the following conditions are equivalent:
\begin{enumerate}
    \item The induced functor $\ps{N}\leftarrow \ps{M}$ is fully faithful, i.e., $\ph$ induces an essential quotient of $\ps{N}$ \label{ConditionQuotientTopos}
    \item $\ph$ is a lax epimorphism in the $2$-category of small categories $\Cat$. \label{ConditionLaxEpi}
    \item the canonical function
    \[M\otimes_{N}M \to M\colon m_0 \otimes_N m_1 \mapsto m_0\ml m_1 \]
    is bijective.\label{ConditionCoend}
    \item For any $m \in M$, $m \otimes_N 1_M = 1_M \otimes m$ in $M\otimes_{N}M$.\label{ConditionTensor}
    \item $\ph$ is epic in the category of monoids. \label{ConditionMonoidsEpi}
    \item $M$is also commutative, and $\ph$ is epic in the category of commutative monoids.\label{ConditionCommutativeMonoidEpi}
\end{enumerate}
\end{proposition}
\begin{proof}
\invmemo{Check here}
    The equivalence between \ref{ConditionQuotientTopos},\ref{ConditionLaxEpi}, and \ref{ConditionCoend} is found in \cite{adamek2001functors}. 
    The implication \ref{ConditionCoend}$\implies$\ref{ConditionTensor} easily follows from $1_M \ast m = m =m \ast 1_M$.

    We prove the implication \ref{ConditionTensor}$\implies$\ref{ConditionMonoidsEpi}. Suppose
    \[
    \begin{tikzcd}
        N\ar[r,"\ph"]&M\ar[r, shift left,"f"]\ar[r,shift right,"g"']&M'
    \end{tikzcd}
    \]
    satisfies $f\circ \ph =g\circ \ph$. Then the function
    \[M\otimes_N M \to M'\colon m_0 \otimes_N m_1 \mapsto f(m_0)\ml g(m_1)\]
    is well-defined. 
    Assuming \ref{ConditionTensor}, 
    $f(m)=g(m)$. 

    We prove \ref{ConditionMonoidsEpi}$\implies$\ref{ConditionCommutativeMonoidEpi}. We can borrow a fact from ring theory:
    if a ring epimorphism $f\colon A \to B$ is epic and $A$ is commutative, then $B$ is also commutative (see \cite{silver1967noncommutative}).
    Suppose $\ph\colon N \to M$ is epic in the category of monoids. Converting $\ph$ to a ring homomorphism, we utilize the monoid ring functor $\Z[-]\colon \mathrm{Monoid}\to \mathrm{Ring}$, which is the left adjoint to the forgetful functor $U$ that forgets addition.
    \[
    \ADJ{\mathrm{Rings}}{U}{\mathrm{Monoids}}{\Z{[-]}}
    \]
    Therefore, the induced ring homomorphism $\Z[\ph] \colon \Z[N] \to \Z[M] $ is epic. Since $\Z[N]$ is commutaive, so are $\Z[M]$ and $M$. It is easy to prove $\ph$ is epic in the category of commutative monoids. 

    At last, we prove \ref{ConditionCommutativeMonoidEpi}$\implies$\ref{ConditionCoend}. This follows from 
    \[
    \begin{tikzcd}
        N\ar[r,"\ph"]\ar[d,"\ph"']&M\ar[d,"\iota_2"']\\
        M \ar[r,"\iota_1"]&M\otimes_N M \arrow[lu, phantom, "\rotatebox{180}{$\lrcorner$}", very near start]
    \end{tikzcd}
    \]
    being a pushout diagram in the category of commutative monoids.
\end{proof}


\begin{remark}[Solid rings]
    Our problem of classifying monoid epimorphisms from $\N$ is analogous to the classification of \emph{solid rings}. A ring $R$ is called \emph{solid}, if the unique ring homomorphism $\Z \to R$ is a ring epimorphism, (which is not necessarily surjective). Some examples include $\Z, \Z/n\Z, \mathbb{Q}, \mathbb{Q} \times \Z/n\Z$. By \cite[Corollary 1.2]{silver1967noncommutative}, such $R$ must be commutative. In \cite{bousfield1972core}, solid rings are introduced and classified.

    Related stories are summarized at \cite{BaezQuestionOnSolidRings} and one will find that the above proposition is quite similar to the case of solid rings.
\end{remark}

\subsubsection{Essential quotients of the topos of discrete dynamical systems}
Summarizing the observations made so far, it is sufficient to determine the monoid epimorphisms from $\N$ to the commutative monoid $M$ in order to determine all essential quotients of $\DD$.

Deferring the detailed calculations and proofs to Appendix C, we will only state the conclusions here. Theorem \ref{TheoremMonoidEpimorphismsFromN} proves that every monoid epimorphism from $\N$ is either the identity morphism $\id_{\N}$ or one of the following two forms. From the viewpoint of the associated discrete dynamical systems (Example \ref{ExampleMonoidDD}), the first case corresponds to $\T{a}{b}$ for $a\in \N$ and $b\neq 0$, and the latter case corresponds to $\T{a}{b}$ for $a\in \N$ and $b= 0$ (Figure \ref{FigureMrab}).

\begin{figure}[ht]
    \begin{shaded}
        \centering
        \begin{tikzpicture} [scale = \scvalue]
        \def \r{2/pi}
        \def \g{-2.3}
        
        \def \a{0}
        \node at (-7, \a -0.5) {$\T{3}{4} \cong (\Mr{3}{4}, - \ast \ph(1)) \colon$};
        \foreach \x in {-3, -2, -1} {
            \draw[black, thick, \arst] (\x, \a) -- (\x + 1, \a);
            \fill[black] (\x,\a) circle (0.06);
        }
        \draw[black, thick, \arst] (0,\a)     arc (90  :  0:\r);
        \draw[black, thick, \arst] (\r,\a-\r) arc (360 :270:\r);
        \draw[black, thick, \arst](0, \a-2*\r)arc (270 :180:\r);
        \draw[black, thick, \arst] (-\r,\a-\r)arc (180 : 90:\r);
        \fill[black] (0,\a) circle (0.06);
        \fill[black] (\r,\a-\r) circle (0.06);
        \fill[black] (0, \a-2*\r) circle (0.06);
        \fill[black] (-\r,\a-\r) circle (0.06);

        \def \c{\g}
        \node at (-7, \c -0.5) {$\T{3}{0} \cong (\Mr{3}{0}, - \ast \ph(1)) \colon$};
        \foreach \x in {-3, -2} {
            \draw[black, thick, \arst] (\x, \c) -- (\x + 1, \c);
        }
        \foreach \x in {-3, -2, -1} {
            \fill[black] (\x,\c) circle (0.06);
        }
        \node[above] at (-2, \c) {$x$};
        \draw[black, thick, \arst] (-1, \c) -- (0, \c -1);
        \foreach \x in {-4, -3, -2, -1, 0} {
            \draw[black, thick, \arst] (\x, \c -1) -- (\x + 1, \c -1);
        }
        \draw[black, thick, \arst] (-4.5, \c -1) -- (-4, \c -1);
        \node at (-5, \c -1) {$\cdots$};
        \draw[black, thick] (1, \c -1) -- (1.5, \c -1);
        \node at (2, \c -1) {$\cdots$};
        \foreach \x in {-4, -3, -2, -1, 0, 1} {
            \fill[black] (\x,\c -1) circle (0.06);
        }
        \node[above] at (-4, \c -1) {$y$};
        
        \end{tikzpicture}
        
        \caption{Visualization of the associated discrete dynamical systems of $\Mr{3}{4}$ and $\Mr{3}{0}$}
        \label{FigureMrab}
    \end{shaded}
\end{figure}
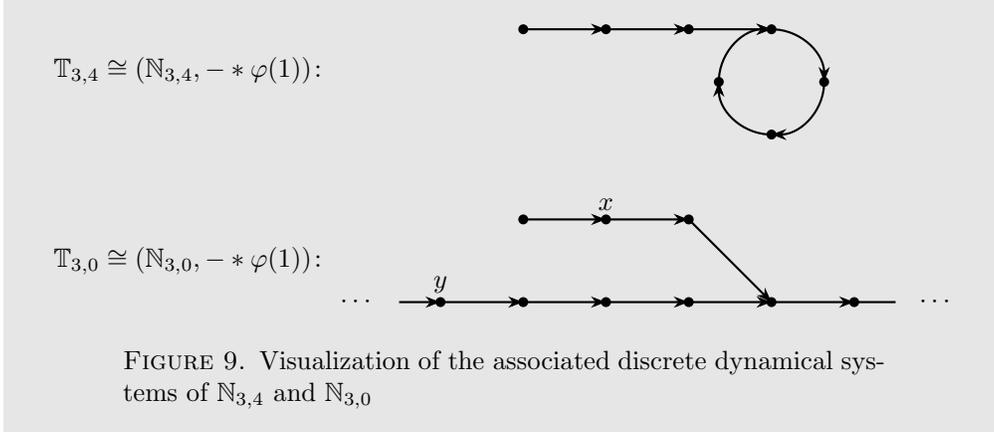

\begin{example}[Case $1/2$, surjective monoid homomorphisms]\label{ExampleCaseOneSurjection}
    The trivial epimorphisms are surjective monoid homomorphisms. For a non-negative integer $a\in \N$ and a positive integer $b\in \N\setminus\{0\}$, we define $\Mr{a}{b}$ as a quotient set divided by the equivalent relation defined as
    \[n\sim m \iff (n=m) \lor ((n,m\geq a) \land (n\equiv m \mod{b})).\]
    
    This set $\Mr{a}{b}$ admits the canonical monoid structure such that the canonical surjection $\ph\colon \N \to \Mr{a}{b}$ is a monoid homomorphism. The associated discrete dynamical system of $(\Mr{a}{b}, \ph(1))$ (Example \ref{ExampleMonoidDD}) is $\T{a}{b}$, visualized as Figure \ref{FigureMrab}.
    Every surjective (epi)morphism is of this form or the identity function $\id_{\N}$.
\end{example}

\begin{example}[Case $2/2$, Integers with a \dq{tail}]\label{ExampleCaseTwoIntegersWithATail}
    There is another type of epimorphism $\ph\colon \N \to \Mr{a}{0}$ such that $\ph$ is injective. 
    Informally, it is the group of integers $\Z (\cong\Z/0\Z )$ (instead of $\Z/b\Z$ for $b>0$) with a \dq{tail} of length $a$.
    The formal definition is as follows. Let $a$ be a non-negative integer. Then, $\Mr{a}{0}$ is defined as a submonoid of the product monoid $\Mr{a}{1} \times \Z$ such as
    \[\Mr{a}{0} \coloneqq \{([i],i)\mid i \in \N\}\cup \{([a],j)\mid j\in \Z\},\]
    where $[i]$ denotes the equivalence class to which $i$ belongs.
    For example, $\Mr{0}{0}$ is the group of integers $\Z$ and $\Mr{1}{0}$ is $\Z$ with a new formal identity element. The homomorphism $\ph$ is defined by $\ph(i) \coloneqq ([i],i)$.

    Just like the former example, the associated discrete dynamical system of $(\Mr{a}{0}, \ph(1))$ (Example \ref{ExampleMonoidDD}) is $\T{a}{0}$, visualized as Figure \ref{FigureMrab}.
\end{example}

Finally, we obtain the classification of essential quotients of $\DD$.
That all of the above monoids define different quotients can be seen by carefully following the construction in the paper \cite{el2002simultaneously}, but here we will prove it by describing the quotients explicitly.
As a preparation, we give an alternative definition of $\Mr{a}{0}$ via generators and relations.
\begin{lemma}\label{LemmaGeneratorandRelations}
For $a\in \N$, the monoid $\Mr{a}{0}$ is isomorphic to the monoid generated by two generators $x,y$ and the following three relations
\[
\gen{x,y \mid xy=yx, x^{a+1}y = x^a, xy^2 = y} 
\]
\end{lemma}
\begin{proof}
    Since $x\coloneqq ([1],1), y \coloneqq ([a],-1) \in \Mr{a}{0}$ satisfy the three equations (see Figure \ref{FigureMrab}), we obtain the induced monoid homomorphism 
    \[f\colon\gen{x,y \mid xy=yx, x^{a+1}y = x^a, xy^2 = y}  \to \Mr{a}{0}.\]
    It is enough to prove this homomorphism is an isomorphism.
    It is surjective, because $\Mr{a}{0}$ is generated by $x$ and $y$.
    The injectivity follows from the fact that every element of the generated monoid is equal to one of the following elements
    \begin{itemize}
        \item $x^k$ for $0\leq k$
        \item $y^k$ for $0<k$
        \item $x^k y$ for $0<k\leq a$.
    \end{itemize}
    \invmemo{See the minimal exponent.}
\end{proof}


\begin{corollary}[Classification of essential quotients]\label{CorollaryClasificationOfEssentialQuotients}
We obtained an isomorphism of posets
\[
\rNN \cup \{\infty\} \cong \mathrm{Quo}(\N) \cong \EQ_{\DD},
\]
where $\mathrm{Quo}(\N)$ denotes the poset of quotient objects of the monoid $\N$, and $\EQ_{\DD}$ denotes the poset of essential quotients of $\DD$.
    To be more specific, for a pair of natural numbers $(a,b)\in \rNN$, the corresponding essential quotient $Q$ is induced by the monoid homomorphism
    \[\ph\colon \N \to \Mr{a}{b}.\]
    More explicitly, a discrete dynamical system $\X=(X,f)$ belongs to $Q$ if and only if 
    \[
    \begin{cases}
        f^{a+b}=f^{a} & (b\neq 0)\\
        \text{the restriction of }f\text{ to }\Image{f^{a}}\text{ is bijective.}& (b=0). 
    \end{cases}
    \]
\end{corollary}

\begin{proof}
    The last explicit description is obtained by Lemma \ref{LemmaGeneratorandRelations}.
    \invmemo{Write more}
\end{proof}

We obtain two immediate corollaries.
\begin{corollary}\label{CorollaryEssentialorHyperconnected}
    Every quotient of $\DD$ is either hyperconnected or essential.
\end{corollary}

\begin{corollary}\label{CorollaryEssentialsAreDense}
If two quotients of $\DD$ contain exactly the same essential quotients, then they are equal to each other.
\end{corollary}


\appendix
\section{Embeddings and Dense order-preserving functions}\label{AppendixSectionDensity}
In this appendix, we briefly recall the notion of embeddings and dense order-preserving functions.

\begin{definition}[Embedding]
    An order-preserving fuction $f\colon P\to Q$ between two posets $P,Q$ is an \emph{embedding} if for any $p,p'\in P$
    \[p\leq p' \iff f(p) \leq f(p').\]
\end{definition}

The following two lemmas are easily proven.
\begin{lemma}
    An embedding is injective.
\end{lemma}

\begin{lemma}\label{LemmaEmbeddingSupremum}
    An embedding reflects supremums. In detail, for an order-preserving function $f\colon P \to Q$, if $f(p)$ is a supremum of $\{f(p_\lambda)\}_{\lambda\in \Lambda}$, then $p$ is a supremum of $\{p_\lambda\}_{\lambda\in \Lambda}$.
\end{lemma}

\begin{definition}[Density of an order-preserving function]\label{DefinitionDense}
    Let $\Dw{P}$ denote the poset of downward closed subsets of a poset $P$. An order-preserving function $f: P\to Q$ between two posets (or possibly large classes of elements) $P, Q$ is said to be \textit{dense} if the function
    \[y^{f}\colon Q \ni q \mapsto \{p\in P\mid f(p)\leq q\}\in \Dw{P}\]
    is embedding.
\end{definition}

\begin{example}
    For a subset $P \subset \mathbb{R}$, the embedding $\iota\colon P \to \mathbb{R}$ is dense if and only if $P$ is dense as a subspace of $\mathbb{R}$, as a topological space with the usual topology. In the case where $P=\mathbb{Q}$, 
    the associated function $y^{\iota}\colon \mathbb{R}\to \Dw{\mathbb{Q}}$ is what's known as the Dedekind cut.
\end{example}

\begin{proposition}
    For an order-preserving function, the following conditions are equivalent.
    \begin{enumerate}
        \item $f$ is dense
        \item For every $q\in Q$, $q$ is the supremum of $f(\{p\in P\mid f(p)\leq q\})$($=f(y^{f}(q))$).
    \end{enumerate}
\end{proposition}
For those who know Kan extension, the second condition is equivalent to saying that 
\[
\begin{tikzcd}
    P\ar[rr,"f"]\ar[rd,"f"']&\ar[d,Rightarrow,"\id_{f}"]&Q\\
    &Q\ar[ru,"\id_{Q}"']&
\end{tikzcd}
\] is a pointwise left Kan extension.

The next proposition is easy to prove by Definition \ref{DefinitionDense}.
\begin{proposition}\label{PropositionSmallnessComparison}
    If $f:P\to Q$ is dense, then $\abs{Q}\leq 2^{\abs{P}}$. 
    In particular, for a dense and injective $f\colon P\to Q$ between possibly large $P$ and $Q$, $P$ is small if and only if $Q$ is small.
\end{proposition}

\section{Some classes of Geometric morphisms}\label{AppendixGeometricMorphisms}
In this appendix, we will briefly recall the notion of geometric morphisms and some classes of them. For details, see \cite{johnstone2002sketchesv1, johnstone2002sketchesv2}.

\begin{definition}[Geometric morphism]
    A \emph{geometric morphism} $f$ from a topos $\E$ to another topos $\F$ is a pair of adjoint functors
    \[
\begin{tikzcd}[column sep = 10pt]
    \E \ar[rr, shift right=6pt, "f_{\ast}"']&\rotatebox{90}{$\vdash$}&\F,\ar[ll, shift right =6pt, "f^{\ast}"']
\end{tikzcd}
\]
whose left adjoint preserves finite limits.
\end{definition}

There are many classes of geometric morphisms, defined in various contexts in topos theory. In this paper, our main focus has been on \emph{connected} geometric morphism:
\begin{definition}[Connected geometric morphism]\label{DefinitionConnectedGM}
    A geometric morphism $f\colon \E\to \F$ is \emph{connected}, if the left adjoint $f^{\ast}$ is fully faithful.
\end{definition}
As explained in subsection \ref{subsectionDefQuotients}, a quotient of a topos is a connected geometric morphism from a topos.

What we will discuss in this section are two classes of geometric morphisms, \emph{hyperconnected} and \emph{essential} geometric morphisms.

\begin{definition}[Hyperconnected geometric morphisms]\label{DefinitionHyperconnectedQuotients}
    A geometric morphism $f\colon \E\to \F$ is \emph{hyperconnected} if $f$ is connected and the essential image of $f^{\ast}$ is closed under taking subobjects.
\end{definition}


\begin{definition}[Essential geometric morphisms]\label{DefinitionEssentialQuotients}
    A geometric morphism $f\colon \E\to \F$ is \emph{essential} if $f^{\ast}$ has a left adjoint.
\end{definition}
\section{\texorpdfstring{Monoid epimorphisms from $\N$}{Monoid epimorphisms from the monoid of natural numbers}}\label{AppendixMonoidEpimorphismsFromN}
\invmemo{Write on ring theory}
The goal of this appendix is to prove Theorem \ref{TheoremMonoidEpimorphismsFromN}, which classifies all monoid epimorphisms from $\N$. 
Our argument is based on researches on monoid epimorphisms and lax epimorphisms \cite{adamek2001functors, el2002simultaneously, isbell1966epimorphisms, nunes2022lax}.
In Proposition \ref{PropositionEquivalenceOfEpiLaxEpi}, we observed that if $\ph\colon \N \to M$ is a monoid epimorphism, then $M$ is commutative.

Let us start by recalling two properties of an element of a (commutative) monoid $x\in M$.
\begin{description}
    \item[Invertible] For any $y \in M$, there exists $y'\in M$ such that $y'x=y$. 
    (Considering $y=1_M$, this is equivalent to the existence of an inverse element.)
    \item[Absorbing] For any $y\in M$, $yx=x$.
\end{description}
The next proposition utilizes weaker variants of the above two propositions, which we call \textit{\qi} and \textit{\qa}.
\begin{proposition}
    Let $M$ be a commutative monoid, $\ph\colon \N \to M$ be a monoid homomorphism, and $x$ be $\ph (1)$. Then $\ph$ is an epimorphism if and only if $x$ satisfies the following two properties.
    \begin{description}
        \item[\Qi] For any $y \in M\setminus \{1_{M}\}$, there exists $y'\in M$ such that $y'x=y$.
        \item[\Qa] For any $y\in M$, there are $n,m\in \N$ such that $yx^n=x^m$.
    \end{description}
\end{proposition}
\begin{proof}
    We prove that it is equivalent to the condition \ref{ConditionTensor} in Proposition \ref{PropositionEquivalenceOfEpiLaxEpi}.
    First, we will prove \dq{if} part. 
    Take an arbitrary element $y \in M$.
    If $y\in \ph(\N)$, then this is easy.

    Suppose $y\notin \ph(\N)$. Then by the assumptions, we can take $a,b\in \N$ such that $yx^a = x^b$ and $y_0 \in M$ such that $y_0 x^b = y$. Then we have
    \begin{align*}
        y\tn 1_M 
            &=y_0 x^b \tn 1_M = y_0\tn  x^b \\
            &=y_0 \tn x^{a+b}y_0  = y_0  x^{a+b} \tn y_0 \\
            &=x^{b} \tn y_0  =1_{M} \tn x^{b} y_0 \\
            &=1_{M} \tn y .
    \end{align*}
    \invmemo{This proof can be intuitively visualized, but it's a bit hard to draw...}

    Next, we will prove \dq{only if} part. Consider the graph $G_M$ with the vertices set $M\times M$ and edges between 
    \[
    \begin{tikzcd}
        (y,xz)\ar[r,dash]&(yx,z)
    \end{tikzcd}
    \]
    for each $y,z \in M$. By the definition of $M \otimes_N M$ (Definition \ref{DefinitionTensor}), the connected components of $G_M$ correspond to the elements of $M \otimes_N M$.
    \begin{description}
        \item[\Qi] Take an arbitrary $y\in M\setminus \{1_M\}$.
        Since $(y, 1_M)$ and $(1_M,y)$ are different vertices in the same connected component, there exists at least one edge that is adjacent to $(y,1_M)$. This implies that either $x$ is invertible, or there exists $y'$ such that $y'x=y$.
        \item[\Qa] Consider an equivalence relation $\approx$, defined by
        \[y\approx y' \iff \exists n,m \in \N,\ yx^n = y'x^m.\]
        If $(y,z)$ and $(y',z')$ are in the same connected component, we can prove $y \approx y'$, by the induction of the length of the shortest path between two vertices.
        In particular, since $y\tn 1_M = 1_M \tn y$, we have $y \approx 1_M$.
    \end{description}
\end{proof}

\begin{remark}\label{RemarkQAisConectedness}
    For a monoid $M$ and its element $x\in M$, $x$ is \qa, if and only if its associated discrete dynamical system $(M,-\ast x)$ is connected.
\end{remark}

\begin{lemma}\label{LemmaInjOrSurj}
    Every epimorphism $\ph\colon \N \to M$ is either injective or surjective.
\end{lemma}
\begin{proof}
    
    Suppose that an epimorphism $\ph$ is neither injective nor surjective. We deduce a contradiction. As before, $\ph(1)$ is denoted by $x$.

    Because $\ph$ is not surjective, we can take $y\in M \setminus \Image{\ph}$. Since $x$ is \qi, and $y$ cannot be written as $x^n (=\ph(n))$, we can take an infinite sequence $y_0, y_1, y_2, \dots \in M$ such that
    \begin{itemize}
        \item $y_0 =y$
        \item $y_{n+1} x = y_{n}$.
    \end{itemize}

    Since $\ph$ is not injective, we can take $a\geq 0$ and $b>0$ such that $x^a=x^{a+b}$. Then we have $y=y_a x^a =y_a x^{a+b} =y x^b$ and $y = y x^{nb}$ for any $n \in \N$. Since $x$ is \qa, there exist $n,m \in \N$ such that $y= y x^{nb} = x^{m}$. This contradicts the assumption that $y \notin \Image{\ph}$.
\end{proof}

With those preparations, we can now classify all epimorphisms from $\N$.

\begin{theorem}\label{TheoremMonoidEpimorphismsFromN}
    Every epimorphism $\ph\colon \N \to M$ is either $\id_{\N}$, of the form of Example \ref{ExampleCaseOneSurjection}, or of the form of Example \ref{ExampleCaseTwoIntegersWithATail}. 
\end{theorem}
\begin{proof}
    If $\ph$ is surjective, then $\ph$ is of the form of Example \ref{ExampleCaseOneSurjection} or $\N$. 
    Suppose $\ph$ is not surjective.
    In this case, by Lemma \ref{LemmaInjOrSurj}, $\ph$ is injective.

    As before, $\ph(1)$ is denoted by $x$.
    Using the fact that $x$ is \qa, we define a function 
    \[\rho \colon M \to \Z\]
     as $\rho(y)=-a+b$, where $a,b\in \N$ satisfies $yx^a=x^b$ (See Figure \ref{FigureRho}). 
              \begin{figure}[ht]
        \begin{shaded}
        \centering
\begin{tikzpicture} [scale = \scvalue]
    \fill[black] (0,1.5) circle (0.06) node[above]{$x^0$};
    \fill[black] (1,1) circle (0.06) node[above]{$x^1$};
    \fill[black] (2,0.5) circle (0.06) node[above]{$x^2$};
    \fill[black] (3,0) circle (0.06) node[above]{$x^3$};
    \fill[black] (4,0) circle (0) node[above]{$x^4$};
    \fill[black] (-5,0) circle (0) node{$\dots$};
    \fill[black] (-4,0) circle (0.06) node[below]{$w_{-4}$};
    \fill[black] (-3,0) circle (0.06) node[below]{$w_{-3}$};
    \fill[black] (-2,0) circle (0.06) node[below]{$w_{-2}$};
    \fill[black] (-1,0) circle (0.06) node[below]{$w_{-1}$};
    \fill[black] (0,0) circle (0.06) node[below]{$w_0$};
    \fill[black] (1,0) circle (0.06) node[below]{$w_1$};
    \fill[black] (2,0) circle (0.06) node[below]{$w_2$};
    \fill[black] (3,0) circle (0.06) node[below]{$w_3$};
    \fill[black] (4,0) circle (0.06) node[below]{$w_4$};
    \fill[black] (5,0) circle (0) node{$\dots$};
    \draw[black, thick] (-4.5,0) -- (4.5,0);
    \draw[black, thick] (0,1.5) -- (3,0) -- (4,0);

    \draw[black, thick, ->>] (0,-0.5)--(0,-1.5);
    \fill[black] (0,-1) circle (0) node[right]{$\rho$};
    \fill[black] (-5,-2) circle (0) node{$\dots$};
    \fill[black] (-4,-2) circle (0.06) node[below]{${-4}$};
    \fill[black] (-3,-2) circle (0.06) node[below]{${-3}$};
    \fill[black] (-2,-2) circle (0.06) node[below]{${-2}$};
    \fill[black] (-1,-2) circle (0.06) node[below]{${-1}$};
    \fill[black] (0,-2) circle (0.06) node[below]{$0$};
    \fill[black] (1,-2) circle (0.06) node[below]{$1$};
    \fill[black] (2,-2) circle (0.06) node[below]{$2$};
    \fill[black] (3,-2) circle (0.06) node[below]{$3$};
    \fill[black] (4,-2) circle (0.06) node[below]{$4$};
    \fill[black] (5,-2) circle (0) node{$\dots$};
    \draw[black, thick] (-4.5,-2) -- (4.5,-2);
\end{tikzpicture}
        \caption{Picture of $M= \Mr{3}{0}$ and $\rho$}
        \label{FigureRho}
        \end{shaded}
    \end{figure}
     This function is well-defined. 
     In fact, if $yx^a = x^b$ and $yx^{a'}=x^{b'}$ and $a'=a+c,\ c\geq 0$, then we have $x^{b'}=yx^{a'} = yx^{a+c}=x^{b+c}$ and hence the injectivity of $\ph$ implies $b'=b+c$ and $-a' + b' = -a + b$. 
     Furthermore, $\rho$ is a monoid homomorphism. It is because if $yx^a =x^b$ and $y'x^{a'}=x^{b'}$, then we have $yy' x^{a+a'}= yx^a y'x^{a'}=x^b x^{b'}= x^{b+b'}$. Here, we used commutativity of $M$. 
     (In fact, $\rho \colon M \to \Z$ is the canonical monoid homomorphism from $M$ to its Grothendieck group $\Z$.)
     

    Notice that $\rho(x^a)=a$. In other words, the following diagram commutes
    \[
    \begin{tikzcd}
        \N\ar[r,"\ph",rightarrowtail]\ar[rd,"\iota"',rightarrowtail]&M\ar[d,"\rho"]\\
        &\Z,
    \end{tikzcd}
    \]where $\iota\colon \N \to \Z$ denotes the inclusion function.

    This homomorphism $\rho$ is surjective because $x$ is \qi, and $\ph$ is not surjective.

     Next, we prove that this monoid homomorphism $\rho$ is injective on $M\setminus \ph(\N)$. Take two elements $y_0, y_1 \in M\setminus \ph(\N)$ such that $\rho(y_0)=\rho(y_1)$. We prove $y_0=y_1$. We can take $a,b \in \N$ such that $y_0x^a =x^b$ and $y_1 x^a = x^b$. Since $x$ is \qi, we can take $z_0, z_1 \in M$ such that $y_0=z_0 x^b$ and $y_1 = z_1 x^b$. (Summarized in the following rough sketch.)
     
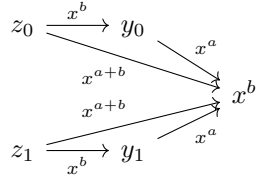
\[
     \begin{tikzcd}[row sep = small]
         z_0\ar[r,"x^b"]\ar[rrd,"x^{a+b}"']&y_0\ar[rd,"x^a"]&\\
         &&x^b\\
         z_1\ar[r,"x^b"']\ar[rru,"x^{a+b}"]&y_1\ar[ru,"x^a"']&
     \end{tikzcd}
     \]
     Then we have $y_0 = z_0 x^b =z_0 y_1 x^a = z_0 z_1  x^{a+b}= z_1 y_0 x^a = z_1 x^b = y_1$.

     We define $w_n \in M$ for each $n \in \Z$, as the unique element of $M$ that satisfies
     

     \begin{enumerate}
         \item $\rho(w_n)=n$ and\label{ConditionRhoN} 
         \item for any $i \in \N$, there exists $y\in M$ such that $yx^{i}=w_n $. \label{ConditionNegStrong}
     \end{enumerate}
     For $n<0$, there exists only one element that satisfies condition \ref{ConditionRhoN}, and this uniqueness implies condition \ref{ConditionNegStrong}. 
     For $n\geq 0$, if there exists such $w_n$, by condition \ref{ConditionNegStrong}, $w_n$ should be $w_{-1} x^{1+n}$. 
     Conversely, if one defines $w_n$ to be $w_{-1} x^{1+n}$, this satisfies both conditions.
     The uniqueness implies that $w_n w_m =w_{n+m}$ and $w_n x^m =w_{n+m}$. Therefore, $\ph(\N)\cup \{w_n\mid n\in \Z\}$ defines a submonoid of $M$ and its multiplication is fully specified.

     Furthermore, since $x$ is \qi, by \dq{dividing} $y$ by $x$ repeatedly, we can prove
     $\ph(\N)\cup \{w_n\mid n\in \Z\} = M$.
     Since $x$ is \qa, we can take the smallest $a\in \N$ such that $w_a = x^a$. Then $M$ is isomorphic to $\Mr{a}{0}$.
\end{proof}


\section*{Acknowledgements}
The authors would like to express our deepest gratitude to the first author's supervisor, Ryu Hasegawa, for his consistent guidance and support. Special thanks are extended to Hisashi Aratake, Yuta Yamamoto, Haruya Minoura, Morgan Rogers, Ivan Toma\v{s}i\'{c}, Luka Ilic, 
Junnosuke Koizumi,
Taichi Yasuda,  
and Takumi Watanabe for their enlightening discussions.

We are also grateful for the discussions at the \dq{mspace topos} and would like to extend our gratitude to its organizers, Toshihiko Nakazawa and Fumiharu Kato for fostering such an encouraging environment.

We would like to thank Koshiro Ichikawa and Yuto Kawase for reading the preprint and pointing out typographical errors and ambiguous sentences. Their remarks enhanced the quality of our paper.

We would like to express deep gratitude to the anonymous reviewer who carefully read the manuscript and provided numerous valuable comments. This greatly assisted the inexperienced authors in enhancing the quality of the paper.

Finally, this research was supported by Forefront Physics and Mathematics Program to Drive Transformation (FoPM), World-leading Innovative Graduate Study (WINGS) Program, the University of Tokyo.

\paragraph{\textbf{Declaration of generative AI and AI-assisted technologies in the writing process}}
During the preparation of this work the authors used chatGPT, DeepL, and Grammarly in order to improve our English. After using this tool/service, the authors reviewed and edited the content as needed and take full responsibility for the content of the publication.
\printbibliography
\end{document}